\documentclass[a4paper, 15pt]{article}
\usepackage[utf8]{inputenc}
\usepackage[T1]{fontenc}
\usepackage[french,english]{babel}    
\usepackage{tikz}
\usetikzlibrary{positioning}

\usepackage[left=2cm,right=2cm,top=2cm,bottom=2cm]{geometry}         		
\usepackage{graphicx}				
\usepackage{amssymb}
\usepackage[backend=bibtex,style=numeric,
          maxnames=1000]{biblatex}
\usepackage{lineno}
\usepackage{amsthm}
\usepackage{amsmath}
\usepackage{amssymb}
\usepackage{amsfonts}
\usepackage{mathtools}
\usepackage{amssymb}
\usepackage{latexsym}
\usepackage{hyperref}
\usepackage{mathrsfs}
\usepackage[utf8]{inputenc}
\usepackage[parfill]{parskip}
\usepackage{xcolor}
\usepackage[affil-it]{authblk}
\usepackage{enumitem}

\usepackage[toc]{appendix}

\usepackage[left=2cm,right=2cm,top=2cm,bottom=2cm]{geometry}

\numberwithin{equation}{section}

\newtheorem{defi}{Definition}[section]
\newtheorem{unTheorem}{Theorem}[section]

\newtheorem{propal}[defi]{Proposition}
\newtheorem{rem}[defi]{Remark}
\newtheorem{cor}[defi]{Corollary}
\newtheorem{lem}[defi]{Lemma}
\newtheorem{ass}[defi]{Assumption}
\newtheorem{nota}[defi]{Notation}

\allowdisplaybreaks

\title{High-order long-time asymptotics for small solutions \\ to the one-dimensional nonlinear Schrödinger equation}

\author[1,2]{Jacek Jendrej\thanks{jendrej@imj-prg.fr}}
\author[1]{Tony Salvi\thanks{tony.salvi@imj-prg.fr}}
\affil[1]{
  \begingroup
  \setlength{\parskip}{0pt}
  Institut de Math\'ematiques de Jussieu, Sorbonne Universit\'e, Universit\'e Paris Cit\'e\par
  4 place Jussieu, 75005 Paris, France
  \endgroup
}
\affil[2]{
  \begingroup
  \setlength{\parskip}{0pt}
Faculty of Applied Mathematics, AGH University of Krak\'ow\par
  al. Adama Mickiewicza 30, 30-059 Krak\'ow, Poland
  \endgroup
}
\begin{document}

\maketitle
\begin{center}
    \textbf{Abstract}
    \begin{minipage}{\dimexpr\textwidth-1.0cm}
\vspace{0.5cm}
We investigate the global well-posedness and modified scattering for the one-dimensional Schrödinger equation with gauge-invariant polynomial nonlinearity. For small localized initial data of finite energy in a low-regularity class, we establish global existence of solution together with persistence of the localization of the associated profile. We further provide a rigorous derivation of the asymptotic expansion at arbitrary order of such solutions, taking into account long-range effects induced by the cubic component of the nonlinearity. Our analysis relies on the space-time resonance method.
\end{minipage}
\end{center}
\section{Introduction} 
In this paper, we examine the nonlinear Schrödinger equation (NLS) in $1$ dimension with a polynomial nonlinearity
\begin{equation}
\label{eq:NLS}
    i\partial_tu+\frac{1}{2}\Delta u=\sum_{n=1}^{d}\lambda_n|u|^{2n}u
\end{equation}
such that there exists $n\in\{1,\dots,d\}$ with $\lambda_n\neq0$,
and we consider the Cauchy problem for initial data $u_1$ at $t=1$, where we set 
$u_1=e^{i\Delta/2}f_1$. We assume that the initial data belongs to weighted Sobolev spaces as defined below.\\
\begin{nota}
The usual Sobolev spaces on $\mathbb{R}$ are written as $W^{k,p}_x$. In the particular case where $p=2$, we use the notation $H^k_x$. When working in the frequency variable, we write them as $W^{k,p}_\xi$ and $H^{k}_\xi$ respectively.\\\\
The weighted Sobolev spaces $H^{k,m}_x$ are defined by $H^{k,m}_x:=\{\phi\in S'|\;||\phi||_{H^{k,m}_x}=||\langle x\rangle^m\langle\nabla\rangle^k\phi||_{L^2_x}<\infty\}$. For $k\in\mathbb{N}$, we also have the equivalence $||\phi||_{H^{k,m}_x}\sim\left(\sum_{j=0}^{k}||\langle x\rangle^m\partial^j\phi||^2_{L^2_x}\right)^{1/2}$.\\
The Bochner-Lebesgue space $L^\infty_tX$, for $X$ a Banach space, denotes $L^\infty([1,\infty),X)$.
\end{nota}
In particular, the initial data must have finite energy $||u_1||_{H^{1,0}_x}< \varepsilon_0$ and be localized $||f_1||_{H^{0,2N+1}_x}<\varepsilon_0$, for some $\varepsilon_0>0$ sufficiently small. Then, our result concerns the asymptotic behavior of the solution $u$, corresponding to such initial data, at any order when $t\to\infty$. \\
\begin{unTheorem}
\label{unTheorem:mainTH}
There exists $\varepsilon_0>0$ such that, if $||u_1||_{H^{1,0}_x}+||f_1||_{H^{0,2N+1}_x}<\varepsilon_0$, then there exists a unique global solution to \eqref{eq:NLS}, such that $u\in C([1,\infty),H^1_x)$. Moreover this solution has the sharp decay 
\begin{align}
||u(t)||_{L^\infty_x}\lesssim \frac{C(\varepsilon_0)}{(1+t)^{1/2}}
\end{align}
and admits an asymptotic expansion of order $N$ under the form
\begin{align}
u(t,x)=\frac{e^{\frac{ix^2}{2t}-i\lambda_1|u_{0,0}\left(\frac{x}{t}\right)|^2\ln(t))-i\varphi(\frac{x}{
t})}}{(it)^{1/2}}\sum_{p=0}^N\sum^{2p}_{k=0}\frac{\ln(t)^k}{t^p}u_{p,k}\left(\frac{x}{t}\right)+u_{err.}(t,x)
\end{align}
where for $0\leq p\leq N$, $0\leq k\leq 2p$, $u_{p,k}\in W^{2(N-p),\infty}_x$, $\varphi\in  W^{2N,\infty}_x$ and $||u_{err.}(t)||_{L^\infty_x}=O(t^{-N-1/2-\delta})$ for some $\delta>0$.
\end{unTheorem}
This result may be viewed as an extension of the result of \cite{zbMATH06033880}, which establishes the first-order asymptotic expansion for the one-dimensional cubic case as $t\to+\infty$, to more general polynomial nonlinearities, as well as to the justification of asymptotic expansions at arbitrary order. Both the proof in \cite{zbMATH06033880} and the arguments developed in the present paper rely on the space-time resonance method originally introduced in \cite{zbMATH05528925}, \cite{zbMATH05590422}, \cite{germain2010globalexistencecoupledkleingordon}, and shortly presented in \cite{germain2011spacetimeresonances}. As in \cite{zbMATH06033880}, the asymptotic analysis is simplified by exploiting a specific gauge-invariant structure of the Schrödinger equation associated with the class of nonlinearities considered in \eqref{eq:NLS}, namely nonlinearities of the form $\lambda_n|u|^{2n}u$. In particular, when $\lambda_1\neq0$ in \eqref{eq:NLS}, that is, when a cubic nonlinearity is present as in \cite{zbMATH06033880}, long-range effects must be taken into account, leading to modified scattering.\\\\
To briefly review the historical development of this topic, some of the earliest works investigating (modified) scattering for the one-dimensional cubic nonlinear Schrödinger equation include \cite{6b42dcf7d8904df8ad13134b93af053b}, \cite{10.1063/1.522967} and \cite{1976ZhETF..71..203Z}. In particular, the asymptotic form is derived in \cite{10.1063/1.522967} (at arbitrary order) and in \cite{1976ZhETF..71..203Z}, though without error estimates. Subsequently, in \cite{zbMATH00722094}, high-order asymptotic expansions are given and error bounds are provided under the assumption that the initial data belongs to the Schwartz class, see also [6]. These results rely on the complete integrability of the cubic NLS and make use of the nonlinear steepest descent method developed in \cite{deift1992steepestdescentmethodoscillatory}. In the seminal work \cite{zbMATH01981617}, the regularity assumption on the initial data is relaxed to weighted Sobolev spaces, although only the first-order term is obtained. This approach is further refined in \cite{dieng2008longtimeasymptoticsnlsequation}, with the $\overline{\partial}$-nonlinear steepest method, see also \cite{dieng2018dispersiveasymptoticslinearintegrable}. It is later extended to compute the second-order term in \cite{liu2024higherorderasymptoticsnonlinear}.\\\\
In parallel, the seminal paper \cite{zbMATH01192427} provides the first rigorous proof of modified scattering for the cubic NLS without resorting to complete integrability. The analysis in  \cite{zbMATH01192427} applies to nonlinearities consisting of a cubic term plus a higher-order term (and also in dimensions 2 and 3 with the corresponding critical power nonlinearity) and for initial data in $H^{\gamma,\gamma}_x$ (for $\gamma>1/2$). In \cite{zbMATH05036894}, similar results are obtained for the one-dimensional case with a cubic plus quintic nonlinearity, using a more physical-space-based approach but requiring higher regularity. That work also establishes asymptotic completeness for decaying scattering data and constructs high-order asymptotic expansions from the scattering data (for the final-state problem). The space-time resonance method is then adapted to the one-dimensional cubic NLS in \cite{zbMATH06033880}, yielding a simpler proof than that of \cite{zbMATH01192427}. The arguments from the space-time resonance method are simplified due to the structure of the $|u|^2u$-type nonlinearity and require $H^{1,1}_x$ regularity.  The regularity assumptions are further relaxed in \cite{zbMATH06502053}, which treats initial data in $H^{0,1}_x$
 using a wave-packet (phase-space) localization method. The latter also addresses the case of a nonlinearity that is a sum of a cubic term plus a higher-order term. For a general overview of modified scattering for the one-dimensional cubic NLS, we refer the reader to \cite{zbMATH07623540}. As far as we know, the present work is the only one that addresses high-order asymptotic expansions for polynomial nonlinear Schrödinger equations of the form \eqref{eq:NLS}, even in the strictly cubic case 
($\lambda_n=0$ for all $n>1$) if one looks for methods that must not rely on complete integrability.\\\\
We also note that modified scattering for the one-dimensional cubic nonlinear Schrödinger equation in the presence of a potential was established in \cite{zbMATH06931169,delort:hal-01396705,zbMATH06595439,chen20221dcubicnlsnongeneric,Chen_2022,chen2023longtimedynamicssmallsolutions}. Moreover, the construction of the modified wave operator for small decaying scattering data (i.e. for the final-state problem) was justified in \cite{zbMATH00027419}, and later extended to large non-decaying data in the defocusing case in \cite{arXiv:2506.01871}.
\section{Acknowledgement and funding} 
We thank Gong Chen for suggesting the problem and many helpful discussions.
This work was supported by ERC Grant INSOLIT, No. 101117126.
\tableofcontents
\section{Notation and common results}
We recall the following standard interpolation result.\\
\begin{lem}
\label{lem:interpol}
For $\omega_2\geq\omega\geq\omega_1\geq0$, we have
\begin{align}
||f||_{H^{0,\omega}_x}\leq\left(||f||_{H^{0,\omega_1}_x}\right)^{\frac{\omega_2-\omega}{\omega_2-\omega_1}}\left(||f||_{{H^{0,\omega_2}_x}}\right)^{\frac{\omega-\omega_1}{\omega_2-\omega_1}}.
\end{align}
\end{lem}
We use the following standard Notation.\\
\begin{nota}
\label{nota:index}
We write $\overrightarrow{m}\in\mathbb{N}^n$ for a multi-index of size $n$ and we compute its norm with $|\overrightarrow{m}|:=\sum_{j=1}^n\overrightarrow{m}_j$.\\
We write $\binom{k}{\overrightarrow{m}}:=\frac{k!}{\prod_{j=1}^n\overrightarrow{m}_j!}$ for multinomial coefficients. \\
\end{nota}
\begin{nota}
For a Banach space $X$ and a given time $T>1$, we note $||\phi||_{C^0_TX}$ for $||\phi||_{C^0([1,T],X)}$.\\
\end{nota}
\begin{nota}
We note $\hat{\phi}$ or $\mathcal{F}(\phi)$ the Fourier transform of $\phi$ defined by 
\begin{align}
\hat{\phi}(\xi):=\frac{1}{(2\pi)^{1/2}}\int e^{-ix\xi}\phi(x)dx.
\end{align}
With this normalisation, we have $||\hat{\phi}||_{L^2_\xi}=||\phi||_{L^2_x}$.\\
\end{nota}
\begin{nota}
The symbol $\lesssim$ denotes $\leq$ up a to constant $C$ that depends on universal constant or fixed quantities of the problem such as $d$, the degree of the polynomial nonlinearity, and $(\lambda_n)_{n\in\mathbb{N}}$ its coefficient. 
\end{nota}
\section{Main ideas of the proof}
The first step consists in constructing a local solution $u$ to \eqref{eq:NLS} with suitable bounds on the (weighted) Sobolev norms of the profile $f=e^{-it\Delta/2}u$ and the modified profile $\hat{w}$ (which we introduce just after). This construction is carried out in Section 4 and relies on the assumption that the initial data $(u_1,e^{-i\Delta/2}u_1)$ lies in $H^{1,0}_x\times H^{0,2N+1}_x$.\\\\
Next, in a bootstrap argument, we show that we can extend the time of existence for small enough initial data (for $\varepsilon_0$ small enough), this is the content of Section \ref{section:bootstrap}. Within this bootstrap, we also propagate the required bounds on $f$ and $\hat{w}$. Control of these (weighted) Sobolev norms is the key ingredient that allows us to rigorously derive the asymptotic expansion of $\hat{w}$, and then $u$, via a representation formula.\\\\
This strategy has already been employed to obtain the first-order ($N=0$) asymptotics in the cubic case in \cite{zbMATH06033880} (or with a potential in \cite{zbMATH06931169}). In particular, if $\lambda_1\neq0$, modified scattering occurs, i.e., the cubic nonlinearity produces long-range effects and induces a phase shift correction in the asymptotic behavior of $u$. We now sketch these ideas in greater detail, assuming the bootstrap argument holds, and later emphasize the modifications required to handle higher-order expansions and general polynomial nonlinearities.\\\\
First, we introduce the profile 
\begin{equation}
\label{eq:profile}
    f(t)=e^{-it\Delta/2}u(t)
\end{equation}
which satisfies 
\begin{equation}
\label{eq:NLSf}
    i\partial_tf(t)=e^{-it\Delta/2}(\sum_{n=1}^{d}\lambda_n|u(t)|^{2n}u(t)),
\end{equation}
as follows from \eqref{eq:NLS}. We first focus on the cubic term, that is, the only term whose effects are launched already at first order. Thus, we consider
\begin{equation}
\label{eq:showfirstorder}
    \partial_t\hat{f}(t,\xi)=-i\frac{\lambda_1}{t}|\hat{f}(t)|^2\hat{f}(t)-i\mathfrak{R}^1_1(t,\xi)-i\mathfrak{R}^{>1}_0(t,\xi),\\
\end{equation}
with 
\begin{align*}
    &\mathfrak{R}^1_1=\lambda_1\frac{1}{2\pi}\int\frac{1}{s}(e^{-i\eta_1\eta_2/s}-1)\mathcal{F}^{-1}_{\eta}[\mathcal{N}^1](s,\xi,\eta)d\eta_1d\eta_2,\\&\mathfrak{R}^{>1}_0=\sum^d_{n=1}\lambda_n\frac{1}{(2\pi)^n}\int\frac{1}{s^n}e^{-i\eta \cdot Q^n\eta/2s}\mathcal{F}^{-1}_{\eta}[\mathcal{N}^n](s,\xi,\eta)d\eta.
\end{align*}
Here, the term $\mathcal{N}^n$ denotes a $(2n+1)$-product of $\hat{f}$ and $\eta=(\eta_1,\eta_2,\dots,\eta_{2n})$. These terms are defined more precisely in section \ref{section:notation}. In particular, the terms $\mathfrak{R}^1_1$ and $\mathfrak{R}^{>1}_0$ are error terms arising from a stationary phase approximation.\\\\
To deal with the first nonintegrable\footnote{In the bootstrap, we show that the profile $f$ is uniformly bounded in space and time.} cubic term, we look at $\hat{w}(t)=\Theta(t,\xi)\hat{f}(t,\xi)$, for $\Theta(t,\xi)=e^{i\int^t_1\frac{\lambda_1|\hat{f}(s,\xi)|^2}{s}ds}$, which solves 
\begin{equation}
\label{eq:modifprofileevointro1}
    \partial_t\hat{w}(t,\xi)=-i\Theta(t,\xi)\mathfrak{R}^1_1(t,\xi)-i\Theta(t,\xi)\mathfrak{R}^{>1}_0(t,\xi).\\
\end{equation}
Assuming that $||t^{-\alpha}f||_{L^\infty_t H^{0,1}_x}\lesssim C$, we obtain $||\mathfrak{R}^1_1||_{L^\infty_\xi}+||\mathfrak{R}^{>1}_0||_{L^\infty_\xi}=O(t^{-1-\beta})$ for some $\beta>0$. Integrating \eqref{eq:modifprofileevointro1} in time then shows that the family $(\hat{w}(t))_{t\geq1}$ 
 forms a Cauchy sequence in $L^\infty_\xi$, and therefore converges to a limit $\hat{w}_{0,0}\in L^\infty_\xi$ as $t\to\infty$. This argument is detailed in section \ref{subsection:asymptoticmodif}, and in particular in lemma \ref{lem:mainofbasecase}. We may then use the representation formula 
\begin{align}
\label{eq:trickformulauf}
u(t,x)=\frac{1}{(2i\pi t)^{1/2}}\int e^{i|x-y|^2/2t}f(t,y)dy
\end{align}
together with a Taylor expansion to deduce that, under the assumption $||t^{-\alpha}f||_{L^\infty_t H^{0,1}_x}\lesssim C$, we have
  \begin{align}
u(t,x)=\frac{e^{i|x|^2/2t}}{(2i\pi t)^{1/2}}\hat{f}\left(t,\frac{x}{t}\right)+O(t^{-1/2-\beta})
\end{align}
in $L^\infty_x$ for some $\beta>0$. Finally, using the convergence of $\hat{w}$ and a precise analysis of the phase shift $\Theta$ (see section \ref{subsection:asymptoticprof}), we get 
  \begin{align}
u(t,x)=\frac{e^{i|x|^2/2t}}{(2i\pi t)^{1/2}}e^{\frac{ix^2}{2t}-i\lambda_1|\hat{w}_{0,0}\left(\frac{x}{t}\right)|^2\ln(t)-i\varphi(\frac{x}{
t})}\hat{w}_{0,0}\left(\frac{x}{t}\right)+O(t^{-1/2-\beta})
\end{align}
as $t\to+\infty$. This result is established in section \ref{subsection:asymptoticsolut} using an argument adapted from \cite{zbMATH01192427}. \\\\
To go beyond the result of \cite{zbMATH06033880} (or \cite{zbMATH01192427}, \cite{zbMATH05036894} and \cite{zbMATH06502053}, which rely on different methods\footnote{We refer to the short review \cite{zbMATH07623540}.}), we now turn to the next term in the asymptotic expansion, where the quintic nonlinearity begins to play a more significant role. From \eqref{eq:modifprofileevointro1}, we see that if stronger control on the weighted Sobolev norms of $f$ is available, namely if $||t^{-\alpha}f||_{L^\infty_t H^{0,3}_x}\lesssim C$, and if $\hat{w}\in L^\infty_t W^{2,\infty}_x$, then we get the convergence of  $(\hat{w}(t))_{t\geq1}$ in $W^{2,\infty}_\xi$. By pushing the stationary phase approximation in \eqref{eq:modifprofileevointro1} further, we get 
\begin{align}
\label{eq:modifprofileevointro2}
    \partial_t\hat{w}(t)=&\frac{\lambda_1}{t^2}\Theta(t)\left(2\hat{f}(t)|\partial_\xi\hat{f}(t)|^2+\partial_\xi\hat{f}(t)^2\overline{\hat{f}(t)}+\hat{f}(t)^2\overline{\partial_\xi^2\hat{f}(t)}\right)-i\frac{\lambda_2}{t^2}\Theta(t)|\hat{f}(t)|^4\hat{f}(t)\\
   \nonumber &-i\Theta(t)\mathfrak{R}^1_2(t)-i\Theta(t)\mathfrak{R}^2_1(t)-i\Theta(t,\xi)\mathfrak{R}^{>2}_0(t).
\end{align}
In particular, the contribution $-i\frac{\lambda_2}{t^2}\Theta(t,\xi)|\hat{f}(t,\xi)|^4\hat{f}(t,\xi)$ originates from the quintic nonlinearity. Recalling that $\hat{w}(t,\xi)=\Theta(t,\xi)\hat{f}(t,\xi)$ and $\Theta(t,\xi)=e^{i\int^t_1\frac{\lambda_1|\hat{w}(s,\xi)|^2}{s}ds}$, we can rewrite \eqref{eq:modifprofileevointro2} entirely in terms of $\hat{w}(t,\xi)$ only. We have 
 \begin{align}
\label{eq:modifprofileevointro3}
    \partial_t\hat{w}(t)=&\frac{2\lambda_1}{t^2}\hat{w}(t)(|\partial_\xi\hat{w}(t)|^2+|\partial_\xi F(t)\hat{w}(t)|^2+\partial_\xi F(t)\hat{w}(t)\overline{\partial_\xi\hat{w}(t)}+\partial_\xi\hat{w}(t)\overline{\partial_\xi F(t)\hat{w}(t)})\\
   \nonumber &+\frac{\lambda_1}{t^2}\overline{\hat{w}(t)}(\partial_\xi\hat{w}(t)^2+(\partial_\xi F(t)\hat{w}(t))^2+2(\partial_\xi F\hat{w}(t))\partial_\xi\hat{w}(t))\\
   \nonumber&+\frac{\lambda_1}{t^2}\hat{w}(t)^2(\overline{\partial_\xi^2\hat{w}(t)}+2\overline{\partial_\xi F(t)\partial_\xi\hat{w}(t)}+\overline{\partial^2_\xi F(t)\hat{w}(t)})\\
  \nonumber &-i\frac{\lambda_2}{t^2}|\hat{w}(t)|^4\hat{w}(t)\\
   \nonumber &-i\Theta(t)\mathfrak{R}^1_2(t)-i\Theta(t)\mathfrak{R}^2_1(t)-i\Theta(t)\mathfrak{R}^{>2}_0(t)
\end{align}
for $F(t)=i\int^t_1\frac{\lambda_1|\hat{w}(s)|^2}{s}ds$. Using the fact that $\hat{w}(t,\xi)=\hat{w}_{0,0}(\xi)+O(t^{-\beta})$ in $W^{2,\infty}_\xi$, we deduce that\footnote{We use the same $\beta>0$ without loss of generality.} $F(t)=F_{0,1}\ln(t)+F_{0,0}+O(t^{-\beta})$ and that $\Theta(t)\mathfrak{R}^1_2(t)+\Theta(t)\mathfrak{R}^2_1(t)+\Theta(t)\mathfrak{R}^{>2}_0(t)=O(t^{-2-\beta})$ in $W^{2,\infty}_\xi$.  This yields 
\begin{align}
\label{eq:modifprofileevointro4}
\partial_t\hat{w}(t)&=\frac{\lambda_1}{t^2}\hat{W}_1(\hat{w}_{0,0})-i\frac{\lambda_2}{t^2}|\hat{w}_{0,0}|^4\hat{w}_{0,0}+\frac{\lambda_1^2\ln(t)}{t^2}\hat{W}_2(\hat{w}_{0,0})+\frac{\lambda_1^3\ln(t)^2}{t^2}\hat{W}_3(\hat{w}_{0,0})+O(t^{-2-\beta})
\end{align}
where $\hat{W}_1$, $\hat{W}_2$ and $\hat{W}_3$ are given in their exact form in the appendix \ref{section:appendix}.\\\\
Then, we integrate \eqref{eq:modifprofileevointro4} between $t$ and $+\infty$ to get the asymptotic expansion
\begin{align}
\label{eq:modifprofileevointro5}
&\hat{w}(t)=\hat{w}_{0,0}+\frac{1}{t}\hat{w}_{1,0}+\frac{\ln(t)}{t}\hat{w}_{1,1}+\frac{\ln(t)^2}{t}\hat{w}_{1,2}+O(t^{-1-\beta}).
\end{align}
Knowing that $\hat{f}(t,\xi)=e^{-F(t,\xi)}\hat{w}(t,\xi)$, with $F(t)=i\int^t_1\frac{\lambda_1|\hat{w}(s)|^2}{s}ds$, we also get the expansion for $\hat{f}(t,x)$ as follows 
\begin{align}
\label{eq:profileexpintro}
&\hat{f}(t)=\hat{f}_{0,0}+\frac{1}{t}\hat{f}_{1,0}+\frac{\ln(t)}{t}\hat{f}_{1,1}+\frac{\ln(t)^2}{t}\hat{f}_{1,2}+O(t^{-1-\beta})
\end{align}
with the regularity $\hat{f}_{0,0}\in W^{2,\infty}_\xi$ and $\hat{f}_{1,0},\hat{f}_{1,1},\hat{f}_{1,2}\in L^{\infty}_\xi$. We see how we are able to compute the expansion of $\hat{w}$ and $\hat{f}$ at order two in $L^\infty_\xi$, which yields the expansion for $u$ with the representation formula \eqref{eq:trickformulauf} and a Taylor expansion, from the expansion of $\hat{w}$ and $\hat{f}$ at order one in $W^{2,\infty}_\xi$.\\\\
Then, we proceed by induction and extend this strategy to any order $n\leq N$ as long as we have the control $||t^{-\alpha}f||_{L^\infty_t H^{0,2N+1}_x}+||\hat{w}||_{L^\infty_t W^{2N,\infty}_\xi}\lesssim C$, for some $0<\alpha$ small enough and some $0<C$.
\section{Preliminary computations}
\label{section:notation}
In this section, we give essential notations and some basic estimates that will be useful throughout the rest of the paper, as well as some important remarks.
\subsection{Formulation of the equations}
\label{subsection:formulation}
\begin{nota}
We note
\begin{equation}
\label{eq:Cterm}
    P(u)(t,x)=\int^t_1e^{-is\Delta/2}(\sum_{n=1}^{d}\lambda_n|u(s,x)|^{2n}u(s,x))ds,
\end{equation}
so that, 
\begin{equation}
\label{eq:NLSduhamf}
    f(t,x)=u_\star(x)-iP(u)(t,x).
\end{equation}
\end{nota}
Looking at the Fourier transform of $f$, we obtain 
\begin{equation}
\label{eq:NLSduhamfFourier}
    \hat{f}(t,\xi)=\hat{u}_\star(\xi)-i\hat{P}(u)(t,\xi).\\
\end{equation}
\begin{nota}
\label{nota:general}
From now on, we omit the dependence on $u$ to lighten the notation. We decompose the integral in time as follows:
\begin{equation}
    \hat{P}(t,\xi)=\sum_{n=1}^{d}\lambda_n\hat{P}^n(t,\xi),\\
\end{equation}
for 
\begin{align}
    \hat{P}^n(t,\xi)=\frac{1}{(2\pi)^n}\int^t_1\int e^{is\xi^2/2}\hat{u}(\xi-\eta_{2n})\overline{\hat{u}(\eta_{2n-1}-\eta_{2n})}\prod^{n-1}_{i=1}\hat{u}(\eta_{2i+1}-\eta_{2i})\overline{\hat{u}(\eta_{2i-1}-\eta_{2i})}\hat{u}(\eta_1)d\eta_1\dots d\eta_{2n}ds,
\end{align}
 This can also be written in term of the Fourier transform of the profile $\hat{f}$, as
\begin{equation}
    \hat{P}^n(t,\xi)=\frac{1}{(2\pi)^n}\int^t_1\int e^{is\Phi^n(\xi,\eta)}\hat{f}(\xi-\eta_{2n})\overline{\hat{f}(\eta_{2n-1}-\eta_{2n})}\prod^{n-1}_{i=1}\hat{f}(\eta_{2i+1}-\eta_{2i})\overline{\hat{f}(\eta_{2i-1}-\eta_{2i})}\hat{f}(\eta_1)d\eta_1\dots d\eta_{2n}ds\\
\end{equation}
with the phase
\begin{align*}
   \Phi^n(\xi,\eta)=\frac{1}{2}(\xi^2-(\xi-\eta_{2n})^2+(\eta_{2n-1}-\eta_{2n})^2+\sum^{n-1}_{i=1}(-(\eta_{2i+1}-\eta_{2i})^2+(\eta_{2i}-\eta_{2i-1})^2)-\eta_1^2)
\end{align*}
where $\eta=(\eta_1,\eta_2,\dots,\eta_{2n})$. We will also denote $d\eta_1\dots d\eta_{2n}$ with $d\eta$ as long as it is not confusing. Finally, we write $\eta_{<j+1}=(\eta_1,\eta_2,\dots,\eta_{j})$ the vector containing only the first $j$ component of the Fourier variable. \\
\end{nota}
\begin{rem}
We observe that 
\begin{align*}
   \Phi^n(\xi,\eta)=\eta_{2n}(\xi-\eta_{2n-1})+\sum^{n-1}_{i=1}\eta_{2i}(\eta_{2i+1}-\eta_{2i-1}).
\end{align*}
\end{rem}
\begin{nota}
\label{nota:Pfraknota}
We set 
\begin{align}
    \hat{\mathfrak{P}}^n(s,\xi)=\frac{1}{(2\pi)^n}\int e^{is\Phi^n(\xi,\eta)}\hat{f}(s,\xi-\eta_{2n})\overline{\hat{f}(s,\eta_{2n-1}-\eta_{2n})}\prod^{n-1}_{k=1}\hat{f}(s,\eta_{2k+1}-\eta_{2k})\overline{\hat{f}(s,\eta_{2k-1}-\eta_{2k})}\hat{f}(s,\eta_1)d\eta, 
\end{align}
such that $\hat{P}^n(t,\xi)=\int^t_1\hat{\mathfrak{P}}^n(s,\xi)ds$. \\
\end{nota}
\begin{rem}
\label{rem:noxiinphase}
In fact, as shown in the following lemma, the term $\hat{\mathfrak{P}}^n(s,\xi)$ (and hence $\hat{P}^n(t,\xi)$) can always be written in such a way that the phase does not depend on
$\xi$. This remark is also made in \cite{zbMATH06033880} for the cubic case. It is crucial for our analysis since propagating weighted norms of $f$ ultimately reduces to establishing energy estimates for $\partial_k\hat{f}(\xi)$. In particular, to control $\partial_k\hat{f}(\xi)$, one does not have to account for derivatives in $\xi$ falling on the phase in equation \eqref{eq:NLSduhamfFourier}, which would otherwise introduce an additional weight in $s$ for each derivative.\\
\end{rem}
\begin{lem}
\label{lem:phasenoxi}
For any $n\in\mathbb{N}^\star$, we have 
\begin{align}
\label{eq:phasenoxi}
    \hat{\mathfrak{P}}^n(s,\xi)=\frac{1}{(2\pi)^n}\int e^{is\Psi^n(\eta)}\hat{f}(s,\xi-\eta_{2n})\overline{\hat{f}(s,H^n_{n}-\eta_{2n})}\left(\prod^{N-1}_{i=1}\hat{f}(s,H^n_{i+1}-\eta_{2i})\overline{\hat{f}(s,H^n_{i}-\eta_{2i})}\right)\hat{f}(s,H^n_1)d\eta.
\end{align}
for 
\begin{align*}
    \Psi^n(\eta)=\sum^{n}_{i=1}\eta_{2i}\eta_{2i-1}=\frac{1}{2}\eta\cdot Q^n\eta
\end{align*}
with 
\begin{align*}
   Q=
\begin{pmatrix} 
	0 & 1 & & &\dots &  \\
	1 & 0 & & & & \vdots \\
	& &  &\ddots & & &\\
    \vdots &  &  &  &0 & 1\\
	& \dots & & & 1 & 0\\
	\end{pmatrix}
\end{align*}
and 
\begin{align*}
    H^n_{i}(\xi,\eta)=\xi-\sum^n_{k=i}\eta_{2i-1}.
\end{align*}
\end{lem}
\begin{nota}
\label{nota:Nonlinearity}
We write 
\begin{align}
    \mathcal{N}^n(\xi,\eta)=\hat{f}(\xi-\eta_{2n})\overline{\hat{f}( H^n_{n}-\eta_{2n})}\prod^{n-1}_{i=1}\hat{f}(H^n_{i+1}-\eta_{2i})\overline{\hat{f}(H^n_{i}-\eta_{2i})}\hat{f}(H^n_1)
\end{align}
for the integrand of $\hat{\mathfrak{P}}^n(\xi)$ without the phase $\Psi^n$.
\end{nota}
\begin{proof}[Proof of Lemma \ref{lem:phasenoxi}]
We get the result by induction on $\hat{\mathfrak{P}}^n$. For $n=1$, we have 
\begin{align*}
    \hat{\mathfrak{P}}^1(\xi)&=\frac{1}{2\pi}\int e^{is\eta_2(\xi-\eta_1)}\hat{f}(\xi-\eta_{2})\overline{\hat{f}(\eta_{1}-\eta_{2})}\hat{f}(\eta_1)d\eta_1d\eta_{2}\\
    &=\frac{1}{2\pi}\int e^{is\eta_{2}\eta_{1}}\hat{f}(\xi-\eta_{2})\overline{\hat{f}(\xi-\eta_{1}-\eta_{2})}\hat{f}(\xi-\eta_1)d\eta_1 d\eta_{2}\\
    &=\frac{1}{2\pi}\int e^{is\eta_{2}\eta_{1}}\hat{f}(\xi-\eta_{2})\overline{\hat{f}(H^1_1-\eta_{2})}\hat{f}(H^1_1)d\eta_1 d\eta_{2}.
\end{align*}
Then for $n>1$, we assume that the equation \eqref{eq:phasenoxi} holds for $\mathfrak{P}^{n-1}$, we set $\eta'_{2n-1}=\xi-\eta_{2n-1}$ (and then forget about the prime) and get\footnote{In particular, $\eta_{<2n-1}=(\eta_1,\eta_2,\dots,\eta_{2N-2})$ is the vector of the first $2n-2$ Fourier variables as defined in notation \ref{nota:general}.}
\begin{align*}
   \hat{\mathfrak{P}}^n(\xi)&=\frac{1}{(2\pi)^n}\int e^{is\Phi^n(\xi,\eta)}\hat{f}(\xi-\eta_{2n})\overline{\hat{f}(\eta_{2n-1}-\eta_{2n})}\prod^{n-1}_{i=1}\hat{f}(\eta_{2i+1}-\eta_{2i})\overline{\hat{f}(\eta_{2i-1}-\eta_{2i})}\hat{f}(\eta_1)d\eta_1\dots d\eta_{2n}\\
    &=\frac{1}{(2\pi)^n}\int e^{is(\eta_{2n}\eta_{2n-1}+\Phi^{n-1}(H^n_n,\eta_{<2n-1}))}\hat{f}(\xi-\eta_{2n})\overline{\hat{f}(H^n_{n}-\eta_{2n})}\mathcal{N}^{n-1}(H^n_{n},\eta_{<2n-1})d\eta_1\dots d\eta_{2n},\\
    &=\frac{1}{(2\pi)}\int e^{is\eta_{2n}\eta_{2n-1}}\hat{f}(\xi-\eta_{2n})\overline{\hat{f}(H^n_{n}-\eta_{2n})}\hat{\mathfrak{P}}^{n-1}(f)(H^n_{n})d\eta_{2n-1} d\eta_{2n}.
\end{align*}
We also have  
\begin{align}
    H^{n-1}_i(H^n_n,\eta_{<2n-1})=H^n_n-\sum^{n-1}_{k=i}\eta_{2i-1}=\xi-\sum^{n}_{k=i}\eta_{2i-1}=H^n_{i},
\end{align}
and 
\begin{align*}
    \eta_{2n}\eta_{2n-1}+\Psi^{n-1}(\eta_{<2n-1})=\Psi^{n}(\eta),
\end{align*}
so that 
\begin{align*}
   \hat{\mathfrak{P}}^n(\xi)=\frac{1}{(2\pi)^n}\int e^{is\Psi^n(\eta)}\hat{f}(\xi-\eta_{2n})\overline{\hat{f}( H^n_{n}-\eta_{2n})}\prod^{n-1}_{i=1}\hat{f}(H^n_{i+1}-\eta_{2i})\overline{\hat{f}(H^n_{i}-\eta_{2i})}\hat{f}(H^n_1)d\eta_1\dots d\eta_{2n}.
\end{align*}
\end{proof}
\begin{nota}
\label{nota:Nonlinearitybis}
Moreover, for any derivatives in $\eta$ of $\mathcal{N}^n(\xi,\eta)$, that is, $(\prod_{j=1}^n\partial^{\overrightarrow{\nu}_j}_{\eta_{2j-1}}\partial^{\overrightarrow{\mu}_j}_{\eta_{2j}})\mathcal{N}^n(\xi,\eta)$, we shorten the notation with 
\begin{align}
    (\prod_{j=1}^n\partial^{\overrightarrow{\nu}_j}_{\eta_{2j-1}}\partial^{\overrightarrow{\mu}_j}_{\eta_{2j}})\mathcal{N}^n(\xi,\eta)=\partial^{\overrightarrow{\nu},\overrightarrow{\mu}}\mathcal{N}^n(\xi,\eta),
\end{align}
where $\overrightarrow{\mu}$ represents derivatives with even coordinate in $\eta$ and  $\overrightarrow{\nu}$ represents derivatives with odd coordinate in $\eta$.
\end{nota}
With equation \eqref{eq:NLSduhamfFourier} and notations \ref{nota:general} and \ref{nota:Pfraknota}, we have 
\begin{equation}
    i\partial_t\hat{f}(t,\xi)=\sum_{n=1}^{d}\lambda_n\hat{\mathfrak{P}}^n(t,\xi).
\end{equation}
For each $\lambda_n\hat{\mathfrak{P}}^n(t,\xi)$, we want to know the stationary phase expansion. \\
\begin{lem}
\label{lem:exactformP}
For any $N\in\mathbb{N}$, the following expansion holds:
\begin{align}
\lambda_n\hat{\mathfrak{P}}^n(t,\xi)=\lambda_n\sum_{k=0}^{N}\frac{1}{t^{k+n}}\hat{\mathfrak{P}}^n_k(t,\xi)+\mathfrak{R}^n_{N+1}(t,\xi)
\end{align}
with 
\begin{align}
\hat{\mathfrak{P}}^n_k(t,\xi)=\frac{i^k}{k!}\sum_{\overrightarrow{m}\in\mathbb{N}^n,|\overrightarrow{m}|=k}\binom{k}{\overrightarrow{m}}
\partial^{\overrightarrow{m},\overrightarrow{m}}\mathcal{N}^n(t,\xi,0)
\end{align}
and 
\begin{align}
\mathfrak{R}^n_{N+1}(t,\xi)=\frac{\lambda_n}{(2\pi)^n}\int\frac{1}{t^n}(e^{-i(\eta\cdot  (Q^n)^{-1}\eta)/2t}-\sum^{N}_{j=0}\frac{(-i\eta\cdot (Q^n)^{-1}\eta)^j}{(2t)^{j}j!})\mathcal{F}^{-1}_{\eta}[\mathcal{N}^n](t,\xi,\eta)d\eta
\end{align}
where $\binom{k}{\overrightarrow{m}}=\frac{k!}{\prod_{j=1}^n\overrightarrow{m}_j!}$, see notations \ref{nota:index} on indexes and notations \ref{nota:Nonlinearitybis} on the nonlinearity $\mathcal{N}$.
\end{lem}
\begin{proof}
From lemma \ref{lem:phasenoxi} and notation \ref{nota:Nonlinearity}, we have that
\begin{align*}
  \hat{\mathfrak{P}}^n(t,\xi)&=\frac{1}{(2\pi)^n}\int e^{it\Psi^n(\eta)}\hat{f}(t,\xi-\eta_{2n})\overline{\hat{f}(t,H^n_{n}-\eta_{2n})}\prod^{n-1}_{i=1}\hat{f}(t,H^n_{i+1}-\eta_{2i})\overline{\hat{f}(t,H^n_{i}-\eta_{2i})}\hat{f}(t,H^n_1)d\eta\\
    &=\frac{1}{(2\pi)^n}\int e^{it\eta\cdot Q^n\eta/2}\mathcal{N}^n(t,\xi,\eta)d\eta\\
    &=\frac{1}{(2\pi)^n}\int \frac{1}{t^n}e^{-i(\eta\cdot (Q^n)^{-1}\eta)/(2t)}\mathcal{F}^{-1}_{\eta}[\mathcal{N}^n](t,\xi,\eta)d\eta\\
    &=\frac{1}{(2\pi)^n}\int\frac{1}{t^n}(\sum^{N}_{j=0}\frac{(-i\eta\cdot (Q^n)^{-1}\eta)^j}{(2t)^{j}j!}+e^{-i(\eta\cdot  (Q^n)^{-1}\eta)/(2t)}-\sum^{N}_{j=0}\frac{(-i\eta\cdot (Q^n)^{-1}\eta)^j}{(2t)^{j}j!})\mathcal{F}^{-1}_{\eta}[\mathcal{N}^n](t,\xi,\eta)d\eta\\
    &=\sum_{k=0}^{N}\frac{1}{t^{k+n}}\hat{\mathfrak{P}}^n_k(t,\xi)+\frac{1}{\lambda_n}\mathfrak{R}^n_{N+1}(t,\xi)
\end{align*}
where 
\begin{align*}
    \hat{\mathfrak{P}}^n_k(t,\xi)&=\frac{1}{(2\pi)^n}\int \frac{1}{2^kk!}(-i\eta\cdot (Q^n)^{-1}\eta)^k\mathcal{F}^{-1}_{\eta}[\mathcal{N}^n](t,\xi,\eta)d\eta\\
    &=\frac{i^k}{k!}\sum_{\overrightarrow{m}\in\mathbb{N}^n,|\overrightarrow{m}|=k}\binom{k}{\overrightarrow{m}}
(\prod_{j=1}^n\partial^{\overrightarrow{m}_j}_{\eta_{2j-1}}\partial^{\overrightarrow{m}_j}_{\eta_{2j}})\mathcal{N}^n(t,\xi,0)\\
 &=\frac{i^k}{k!}\sum_{\overrightarrow{m}\in\mathbb{N}^n,|\overrightarrow{m}|=k}\binom{k}{\overrightarrow{m}}
\partial^{\overrightarrow{m},\overrightarrow{m}}\mathcal{N}^n(t,\xi,0).
\end{align*}
where we use the particular form of $Q^n$ given in \ref{lem:phasenoxi}.
\end{proof}
\begin{rem}
The term $\hat{\mathfrak{P}}^n_k(t,\xi)$ is made of a sum of products. Each such product contains $2n+1$ terms consisting of $n+1$ occurrences of $\hat{f}(t,\xi)$ (or its derivatives) and $n$ occurrences of $\overline{\hat{f}(t,\xi)}$ (or its derivatives). The total number of derivatives appearing in each product, as well as the maximum number of derivatives on each term in the product, is equal to $2k$.\\
\end{rem} 
\begin{rem}
The term $\mathfrak{R}^n_{N+1}(t,\xi)$ is interpreted as an error term. \\
\end{rem} 
\begin{nota}
\label{nota:sumtoterror}
We denote the sum $\sum_{n=N+2}^{d}\lambda_n\hat{\mathfrak{P}}^n(t,\xi)$, or equivalently $\sum_{n=N+2}^{d}\mathfrak{R}^n_{0}(t,\xi)$, with $\mathfrak{R}^{> N+1}_0(t,\xi)$. This error term represents the nonlinearities of degree superior or equal to $2N+3$, which do not contribute to the asymptotic approximation at order $N$. \\
\end{nota}
The exact form of $\partial^{\overrightarrow{m},\overrightarrow{m}}\mathcal{N}^n(t,\xi,0)$, and thus $\hat{\mathfrak{P}}^n_k(t,\xi)$, is given by the next lemma, after we define the configuration for derivatives in the next definition, and the estimates on $\mathfrak{R}^n_{N+1}(t,\xi)$ are established in the subsection that follows. \\
\begin{defi}
\label{defi:config}
For $J\in\mathbb{N}^{n\times m}$, interpreted as a vector of multi-index, and $k\in\{1,\dots,n\}$, the term $J_k$ denotes the multi-index $\overrightarrow{m}\in\mathbb{N}^{m}$ such that, for every $j\in\{1,\dots,m\}$, $\overrightarrow{m}_j=J_{k,j}$.
Then, for $\overrightarrow{\nu}\in\mathbb{N}^n$, we define $\mathscr{J}_{\overrightarrow{\nu}}:=\{J\in\mathbb{N}^{n\times(2n+1)}|\;\forall k\in\{1,\dots,n\},\forall l\in\{1,\dots,2n+1\}, |J_k|=\overrightarrow{\nu}_k,\text{ and if }l> 2k,\;  J_{k,l}=0\}$, and for $\overrightarrow{\mu}\in\mathbb{N}^n$, we define $\mathscr{K}_{\overrightarrow{\mu}}:=\{K\in\mathbb{N}^{n\times 2}|\;\forall k\in\{1,\dots,n\}, |K_k|=\overrightarrow{\mu}_k\}$. Then, we set $\mathscr{L}_{\overrightarrow{\mu},\overrightarrow{\nu}}:=\mathscr{J}_{\overrightarrow{\mu}}\times\mathscr{K}_{\overrightarrow{\nu}}$.\\
\end{defi}
\begin{lem}
\label{lem:exactformN}
We have
\begin{multline}
    \partial^{\overrightarrow{m},\overrightarrow{m}}\mathcal{N}^n(\xi,0)=\sum_{(J,K)\in\mathscr{L}_{\overrightarrow{\mu},\overrightarrow{\nu}}}\left(\prod^n_{j=1}\binom{\overrightarrow{m}_j}{K_j}\binom{\overrightarrow{m}_j}{J_j}(-\partial_\xi)^{K_{j,1}+\sum_{l=1}^{n}J_{l,2j+1}}\hat{f}(\xi)(-\partial_\xi)^{K_{j,2}+\sum_{l=1}^{n}J_{l,2j}}\overline{\hat{f}(\xi)}\right)\times\\(-\partial_\xi)^{\sum_{l=1}^{n}J_{l,1}}\hat{f}(\xi)
\end{multline}
\end{lem}
\begin{proof}
Let $\overrightarrow{\mu},\overrightarrow{\nu}\in\mathbb{N}^n$, we want to compute any derivatives of $\mathcal{N}^n(\xi,\eta)$ in $\eta$, that is, $\partial^{\overrightarrow{\nu},\overrightarrow{\mu}}\mathcal{N}^n(\xi,\eta)$ (see notations \ref{nota:Nonlinearity}).
We recall that 
\begin{align*}
     \mathcal{N}^n(\xi,\eta)&=\hat{f}(\xi-\eta_{2n})\overline{\hat{f}( H^n_{n}-\eta_{2n})}\prod^{n-1}_{i=1}\hat{f}(H^n_{i+1}-\eta_{2i})\overline{\hat{f}(H^n_{i}-\eta_{2i})}\hat{f}(H^n_1)\\
     &=\prod^{n}_{i=1}\hat{f}(H^n_{i+1}-\eta_{2i})\overline{\hat{f}(H^n_{i}-\eta_{2i})}\hat{f}(H^n_1).
\end{align*}
By direct calculations, we have 
\begin{align*}
    \partial^{\overrightarrow{\nu},\overrightarrow{\mu}}\mathcal{N}^n(\xi,\eta)=\sum_{(K,J)\in\mathscr{K}_{\overrightarrow{\mu}}\times\mathscr{J}_{\overrightarrow{\nu}}}(\prod^n_{k=1}\binom{\overrightarrow{\mu}_k}{K_k}\binom{\overrightarrow{\nu}_k}{J_k})\prod^n_{j=1}\mathcal{Q}_j^{K_{j,1},K_{j,2},\sum_{l=1}^{n}J_{l,2j+1},\sum_{l=1}^{n}J_{l,2j}}(\xi,\eta)(-\partial_\xi)^{\sum_{l=1}^{n}J_{l,1}}\hat{f}(H^n_{1})
\end{align*}
where 
\begin{align*}
    \mathcal{Q}_j^{\kappa_1,\kappa_2,\iota_1,\iota_2}(\xi,\eta)=(-\partial_\xi)^{\kappa_1+\iota_1}\hat{f}(H^n_{j+1}-\eta_{2j})(-\partial_\xi)^{\kappa_2+\iota_2}\overline{\hat{f}(H^n_{j}-\eta_{2j})}.
\end{align*}
Here $K$ and $J$ represent configurations of even and odd derivatives. In the particular case of $\overrightarrow{\mu}=\overrightarrow{\nu}$ and $\eta=0$, together with the preceding remark, we recover the stated expression. 
\end{proof}
\subsection{Basic error estimates}
\label{subsection:errorestimate}
We provide the main estimate for the error term $ \mathfrak{R}^n_r(t)$ when $t\to+\infty$.\\
\begin{propal}
\label{propal:controlerror}
Given $p\in\mathbb{N}$. If $||t^{-\alpha}\langle x\rangle^{p+1} f||_{L^\infty_tL^2_x}\leq C$ for some $1/(16d+8)>\alpha>0$ and some $C>0$ then, for $r,q\in\mathbb{N}$ such that $2r+q-2=p$ if $r>0$ or $p=q$ if $r=0$, there exists $1>\delta>0$ such that for $s\geq1$
\begin{align}
&||\mathfrak{R}^n_r(s)||_{W^{q,\infty}_\xi}\lesssim s^{-n-r+1-\delta}C^{2n+1}.
\end{align}
Moreover, if $||\hat{w}||_{L^\infty_tW^{p,\infty}_\xi}\leq C$, then, under the same relations for $q,r$ and $p$, there exists $1>\delta>0$ such that for $s\geq1$
\begin{align}
&||\Theta\mathfrak{R}^n_r(s)||_{W^{q,\infty}_\xi}\lesssim s^{-n-r+1-\delta}h(C)C^{2n+1}
\end{align}
for some polynomial $h$.
\end{propal}
To prove proposition \ref{propal:controlerror}, we first need the two following lemmas. \\
\begin{lem}
\label{lem:Rtrans}
For $s\in[1,+\infty)$, $\xi\in\mathbb{R}$ and $\sigma\in\mathbb{R}^{2n+1}$, we have 
\begin{align}
\label{equation:Rtransform}
    \mathcal{F}^{-1}_{\eta}[\mathcal{N}^n](s,\xi,\sigma)=\frac{1}{(2\pi)^{1/2}}e^{i\xi(\sum^{2n}_{j=1}\sigma_{j})}\int e^{-i\xi x}\mathcal{M}^n(s,x,\sigma)dx
\end{align}
with 
\begin{align*}
   \mathcal{M}^n=\left(\prod^{n}_{j=1}f(x-\sigma_{2j}-\sum^{2(j-1)}_{k=1}\sigma_k-\sum^{n-1}_{k=j}\sigma_{2k+1})\overline{f(x-\sum^{2(j-1)}_{k=1}\sigma_{k}-\sum^{n-1}_{k=j}\sigma_{2k+1})}\right)f(x-\sum^{n-1}_{k=0}\sigma_{2k+1}).
\end{align*}
\end{lem}
\begin{proof}
Let $\mathcal{G}^n(s,\xi,\sigma)=\mathcal{F}^{-1}_{\eta}[\mathcal{N}^n](s,\xi,\sigma)$ and let us now forget about the time variable $s$.
 Before starting the calculation, let us recall that in 1 dimension
\begin{align*}
    \mathcal{F}^{-1}_{\eta}(\hat{g}(\xi-\eta)e^{-i x\eta})(\sigma)&=(2\pi)^{-1/2}\int \hat{g}(\xi-\eta)e^{i\eta(\sigma- x)}d\eta\\
    &=(2\pi)^{-1/2}\int \hat{g}(\eta)e^{i(\xi-\eta)(\sigma- x)}d\eta\\
    &=e^{i\xi(\sigma- x)}g(x-\sigma)
\end{align*}
and 
\begin{align*}
    \mathcal{F}^{-1}_{\eta}(\overline{\hat{g}(\xi-\eta)e^{-i x\eta}})(\sigma)=e^{i\xi(\sigma+ x)}\overline{g(x+\sigma)}.
\end{align*}
We proceed by induction. First, we show that \eqref{equation:Rtransform} holds for $N=1$. We have 
\begin{align*}
    \mathcal{G}^1(\xi,\sigma)&=\mathcal{F}^{-1}_{\eta_1,\eta_2}[\hat{f}(\xi-\eta_{2})\overline{\hat{f}( H^1_{1}-\eta_{2})}\hat{f}(H^1_1)](\sigma)\\
    &=\mathcal{F}^{-1}_{\eta_1,\eta_2}[\hat{f}(\xi-\eta_{2})\overline{\hat{f}( \xi-\eta_1-\eta_{2})}\hat{f}(\xi-\eta_1)](\sigma)\\
    &=\mathcal{F}^{-1}_{\eta_1}[\frac{1}{(2\pi)^{1/2}}\int e^{i\xi(\sigma_2-y)}f(y-\sigma_{2})e^{i(\xi-\eta_1)y}\overline{f( y)}\hat{f}(\xi-\eta_1)dy](\sigma_1)\\
    &=\mathcal{F}^{-1}_{\eta_1}[\frac{1}{(2\pi)^{1/2}}\int e^{i\xi\sigma_2}f(y-\sigma_{2})\overline{f( y)}e^{-i\eta_1y}\hat{f}(\xi-\eta_1)dy](\sigma_1)\\
    &=\frac{1}{(2\pi)^{1/2}}\int e^{i\xi\sigma_2}f(y-\sigma_{2})\overline{f( y)}e^{i\xi(\sigma_1-y)}f(y-\sigma_1)dy\\
    &=\frac{1}{(2\pi)^{1/2}}e^{i\xi(\sigma_2+\sigma_1)}\int e^{-i\xi y}f(y-\sigma_{2})\overline{f( y)}f(y-\sigma_1)dy\\
    &=\frac{1}{(2\pi)^{1/2}}e^{i\xi(\sigma_2+\sigma_1)}\int e^{-i\xi y}\mathcal{M}^1(y,\sigma)dy,
\end{align*}
which is the form of \eqref{equation:Rtransform}.  
Now, assuming that \eqref{equation:Rtransform} holds for $n-1$, we may compute the inverse Fourier transform of $\mathcal{G}^{n-1}$ in $\xi$ easily (as the dependence in $\xi$ is only in the phase). Indeed,
\begin{align*}
    \mathcal{F}^{-1}_{\xi}(\mathcal{G}^{n-1})(y,\sigma_{<2n-1})&=\int \mathcal{M}^{n-1}(x,\sigma_{<2n-1})\delta(y+\sum^{2(n-1)}_{j=1}\sigma_{j}-x)dx\\
    &=\mathcal{M}^{n-1}(y+\sum^{2(n-1)}_{j=1}\sigma_{j},\sigma_{<2n-1}).
\end{align*}
Then, we can observe that $\mathcal{G}^n(\sigma)$ directly depends on $\mathcal{F}^{-1}_{\xi}(\mathcal{G}^{n-1})$ in the following calculations: 
\begin{align*}
    \mathcal{G}^n(\sigma)&=\mathcal{F}^{-1}_{\eta}[\hat{f}(\xi-\eta_{2n})\overline{\hat{f}( H^n_{n}-\eta_{2n})}\left(\prod^{n-1}_{i=1}\hat{f}(H^n_{i+1}-\eta_{2i})\overline{\hat{f}(H^n_{i}-\eta_{2i})}W\right)\hat{f}(H^n_1)](\sigma)\\
    &=\mathcal{F}^{-1}_{\eta_{<2n}}[\frac{1}{(2\pi)^{1/2}}\int e^{i\xi(\sigma_{2n}-y)}f(y-\sigma_{2n})e^{iH^n_ny}\overline{f( y)}\left(\prod^{n-1}_{i=1}\hat{f}(H^n_{i+1}-\eta_{2i})\overline{\hat{f}(H^n_{i}-\eta_{2i})}\right)\hat{f}(H^n_1)dy](\sigma_{<2n})\\
    &=\mathcal{F}^{-1}_{\eta_{2n-1}}[\frac{1}{(2\pi)^{1/2}}\int e^{i\xi\sigma_{2n}}f(y-\sigma_{2n})\overline{f( y)}e^{-i\eta_{2n-1}y}\mathcal{F}^{-1}_{\eta_{<2n-1}}(\left(\prod^{n-1}_{i=1}\hat{f}(H^n_{i+1}-\eta_{2i})\overline{\hat{f}(H^n_{i}-\eta_{2i})}\right)\hat{f}(H^n_1))dy](\sigma_{<2n})\\
     &=\frac{1}{(2\pi)^{1/2}}\int e^{i\xi\sigma_{2n}}f(y-\sigma_{2n})\overline{f( y)}\mathcal{F}^{-1}_{\eta_{2n-1}}[e^{-i\eta_{2n-1}y}\mathcal{G}^{n-1}(\xi-\eta_{2n-1},\sigma_{<2n-1})](\sigma_{2n-1})dy\\
      &=\frac{1}{(2\pi)^{1/2}}\int e^{i\xi\sigma_{2n}}f(y-\sigma_{2n})\overline{f( y)}e^{i\xi(\sigma_{2n-1}-y)}\mathcal{F}^{-1}_{\xi}[\mathcal{G}^{n-1}](y-\sigma_{2N-1},\sigma_{<2N-1})dy\\
       &=\frac{e^{i\xi\sum^{2n}_{j=1}\sigma_{j}}}{(2\pi)^{1/2}}\int e^{-i\xi x}f(x-\sigma_{2n}-\Lambda^{(n-1)})\overline{f( x-\Lambda^{(n-1)}})\mathcal{F}^{-1}_{\xi}(\mathcal{G}^{n-1})(x-\Lambda^{(n-1)}-\sigma_{2n-1},\sigma_{<2n-1})dx
\end{align*}
for $\Lambda^{(n-1)}=\sum^{2(n-1)}_{j=1}\sigma_{j}$. Finally, we get 
\begin{align*}
    \mathcal{G}^n(\sigma)&=\frac{1}{(2\pi)^{1/2}}e^{i\xi\sum^{2n}_{j=1}\sigma_{j}}\int e^{-i\xi x}f(x-\sigma_{2n}-\Lambda^{(n-1)})\overline{f( x-\Lambda^{(n-1)}})\mathcal{M}^{n-1}(x-\sigma_{2n-1},\sigma_{<2n-1})dx\\
    &=\frac{1}{(2\pi)^{1/2}}e^{i\xi\sum^{2n}_{j=1}\sigma_{j}}\int e^{-i\xi x}\mathcal{M}^{n}(x,\sigma)dx
\end{align*}
which is the desired result. 
\end{proof}
\begin{lem}
\label{lem:Rcalculus}
For $\mathcal{M}^n$ defined as in the previous lemma \ref{lem:Rtrans} and some $\alpha\geq0$, we have 
\begin{align}
\label{equation:Rcalculus}
\int\int (\sum_{j=1}^{2n}|\sigma_j|+|x|)^{\alpha}\mathcal{M}^n(x,\sigma)dxd\sigma\lesssim\int (\sum_{j=1}^{2n+1}|Y_{j}|)^{\alpha}\left(\prod^{n}_{j=1}|f(Y_{2j+1})||\overline{f(Y_{2j})}|\right)|f(Y_{1})|dY_1\dots dY_{2n+1}
\end{align}
\end{lem}
\begin{proof}
We show the result for $\alpha=0$ by induction for $n\geq 1$. Up to replacing the equality sign with an inequality sign, cases with $\alpha>0$ can be studied in the same way because all the changes of variable operated in the proof are linear combinations. For $n=1$, we have 
\begin{align*}
\int\mathcal{M}^1(x,\sigma)dxd\sigma&=\int\int\int f(x-\sigma_2)\overline{f(x)}f(x-\sigma_1)dxd\sigma_1d\sigma_2\\
&=\int\int\int f(Y_3)\overline{f(Y_2)}f(Y_1)dY_3dY_2dY_1
\end{align*}
Then, for $n>1$, we assume that the claim holds for $n-1$. We use the notation
\begin{align*}
   \mathcal{M}^n(x,\sigma)=\prod^{n}_{j=1}f(X_{2j+1}(x,\sigma))\overline{f(X_{2j}(x,\sigma))}f(X_{1}(x,\sigma)),
\end{align*}
where, in the odd case,
\begin{align*}
X_{2j+1}(x,\sigma)=x-\sigma_{2j}-\sum^{2(j-1)}_{k=1}\sigma_k-\sum^{n-1}_{k=j}\sigma_{2k+1},
  \end{align*}
  with  $\sigma_{0}=0$ by convention, and in the even case,
  \begin{align*}
X_{2j}(x,\sigma)=x-\sum^{2(j-1)}_{k=1}\sigma_k-\sum^{n-1}_{k=j}\sigma_{2k+1}.
  \end{align*}
Then, observing that only $X_{2n+1}(x,\sigma)$ depends on $\sigma_{2n+1}$ yields
\begin{align*}
\int\int \mathcal{M}^n(x,\sigma)dxd\sigma&=\int\int\left(\prod^{n}_{j=1}f(X_{2j+1}(x,\sigma))\overline{f(X_{2j}(x,\sigma))}\right)f(X_{1}(x,\sigma))dxd\sigma\\
&=\int\int\int f(Y_{2n+1})\overline{f(X_{2n}(x,\sigma))}\mathcal{M}^{n-1}(x-\sigma_{2n-1},\sigma_{<2n-1})dxd\sigma_1\dots d\sigma_{2n-1}dY_{2n+1},\\
&=\int\int\int f(Y_{2n+1})\overline{f(X_{2n}(y+\sigma_{2n-1},\sigma))}\mathcal{M}^{n-1}(y,\sigma_{<2n-1})dyd\sigma_1\dots d\sigma_{2n-1}dY_{2n+1},\\
&=\int\int\int\int f(Y_{2n+1})\overline{f(Y_{2n})}\mathcal{M}^{n-1}(y,\sigma_{<2n-1})dyd\sigma_1\dots d\sigma_{2n-2}dY_{2n}dY_{2n+1},\\
&=\int \left(\prod^{n}_{j=1}f(Y_{2j+1})\overline{f(Y_{2j})}\right)f(Y_{1})dY_1\dots dY_{2n+1}
\end{align*}
using the induction hypotheses for $n-1$. 
\end{proof}
\begin{proof}[Proof of proposition \ref{propal:controlerror}]
For $r>0$, we have
\begin{align*}
    \partial^q_\xi \mathfrak{R}^n_r(s)&=\lambda_n\frac{1}{(2\pi)^n}\int\frac{1}{s^n}(e^{-i(\eta\cdot  (Q^n)^{-1}\eta)/(2s)}-\sum^{r-1}_{j=0}\frac{(-i\eta\cdot (Q^n)^{-1}\eta)^j}{(2s)^{j}j!}) \partial^q_\xi\mathcal{F}^{-1}_{\eta}[\mathcal{N}^n](s,\xi,\eta)d\eta\\
    &\lesssim \int\frac{1}{s^{n+r-\beta}}(\eta\cdot (Q^n)^{-1}\eta)^{r-\beta} |\partial^q_\xi\mathcal{F}^{-1}_{\eta}[\mathcal{N}^n](s,\xi,\eta)|d\eta\\
    &\lesssim\int\int\frac{1}{s^{n+r-\beta}}|\eta|^{2r-2\beta} (\sum_{j=1}^{2n}|\eta_j|+|x|)^{q}|\mathcal{M}^{n}(s,x,\eta)|dxd\eta\\
    &\lesssim\frac{1}{s^{n+r-\beta}}\int\int(\sum_{j=1}^{2n}|\eta_j|+|x|)^{2r+q-2\beta} |\mathcal{M}^{n}(s,x,\eta)|dxd\eta
\end{align*}
with $0<\beta<1$, and for $r=0$
\begin{align*}
    \partial^q_\xi \mathfrak{R}^n_0(s)&=\lambda_n\frac{1}{(2\pi)^n}\int\frac{1}{s^n}(e^{-i(\eta\cdot  (Q^n)^{-1}\eta)/(2s)}) \partial^q_\xi\mathcal{F}^{-1}_{\eta}[\mathcal{N}^n](s,\xi,\eta)d\eta\\
    &\lesssim \frac{1}{s^{n}} \int\int(\sum_{j=1}^{2n}|\eta_j|+|x|)^{q}|\mathcal{M}^{n}(s,x,\eta)|dxd\eta.
\end{align*} 
where we use lemma \ref{lem:Rtrans} to compute $\mathcal{F}^{-1}_{\eta}[\mathcal{N}^n](s,\xi,\eta)$ in both cases. Then, applying lemma \ref{lem:Rcalculus} leads to
\begin{align*}
    \partial^q_\xi \mathfrak{R}^n_r(s)\lesssim\frac{1}{s^{n+r-\beta}}\int\int(\sum_{j=1}^{2n+1}|Y_{j}|)^{2r+q-2\beta}\left(\prod^{n}_{j=1}|f(s,Y_{2j+1})||\overline{f(s,Y_{2j})}|\right)|f(s,Y_1)|dY_1\dots dY_{2n+1}
\end{align*}
and 
\begin{align*}
    \partial^q_\xi \mathfrak{R}^n_0(s)\lesssim\frac{1}{s^{n}} \int\int(\sum_{j=1}^{2n+1}|Y_{j}|)^{q}\left(\prod^{n}_{j=1}|f(s,Y_{2j+1})||\overline{f(s,Y_{2j})}|\right)|f(s,Y_1)|dY_1\dots dY_{2n+1}.
\end{align*}
This gives $2n+1$ integrals with weighted $f$'s as integrands that can be computed one by one. For $r>0$, we get 
\begin{align*}
    \partial^q_\xi \mathfrak{R}^n_r(s)
    &\lesssim \frac{1}{s^{n+r-\beta}}\sum_{\overrightarrow{m}\in\mathbb{N}^{2n+1}, |\overrightarrow{m}|=2r+q}\prod^{2n+1}_{j=1}||\langle x\rangle^{m_j\gamma} f(s)||_{L^1_x} \\
    &\lesssim \frac{1}{s^{n+r-\beta-(2n+1)\alpha}}\sum_{\overrightarrow{m}\in\mathbb{N}^{2n+1}, |\overrightarrow{m}|=2r+q}\prod^{2n+1}_{j=1}||s^{-\alpha}\langle x\rangle^{m_j\gamma+1/2+\varepsilon} f(s)||_{L^2_x} \\
    &\lesssim \frac{1}{s^{n+r-\beta-(2n+1)\alpha}}C^{2n+1}
\end{align*}
for $|\overrightarrow{m}|=\sum_{j=1}^{2n+1}|\overrightarrow{m}_j|$ and $\beta>3/4$, and where $0<\gamma<1$ is such that $2r+q-2\beta=(2r+q)\gamma$ and $\varepsilon>0$ such that $2\beta-3/2>\varepsilon$. With a simpler reasoning, we have 
\begin{align*}
    \partial^q_\xi \mathfrak{R}^n_0(s)
    &\lesssim \frac{1}{s^{n-(2n+1)\alpha}}\sum_{\overrightarrow{m}\in\mathbb{N}^{2n+1}, |\overrightarrow{m}|=q}\prod^{2n+1}_{j=1}||s^{-\alpha}\langle x\rangle^{m_j+1/2+\varepsilon} f(s)||_{L^2_x} \\
    &\lesssim \frac{1}{s^{n-(2n+1)\alpha}}C^{2n+1}
\end{align*}
for $1/2>\varepsilon>0$. Thus, if $\alpha<1/(16d+8)$ and $\beta>3/4$ is small enough, we obtain for some $1>\delta>0$
\begin{align*}
    \partial^q_\xi \mathfrak{R}^n_r(s)\lesssim s^{-n-r+1-\delta}C^{2n+1}
\end{align*}
for $s\geq 1$.
Then, bounds on lower derivatives can be obtain the same way:
\begin{align}
|| \mathfrak{R}^n_r(s)||_{W^{q,+\infty}_x}\lesssim s^{-n-r+1-\delta}C^{2n+1}.
\end{align}
This is the first statement of the proposition. Now, knowing that 
\begin{align*}
    \Theta(s,\xi)=e^{i\int^s_1\frac{|\hat{f}(\tau,\xi)|^2}{\tau}d\tau}
\end{align*}
and that $|\hat{f}(\tau,\xi)|=|\hat{w}(\tau,\xi)|$, we get that, for $k\leq q$,
\begin{align*}
    ||\partial^k_\xi \Theta(s)||_{L^\infty_x}&\lesssim \sum^{k}_{j=1}\ln(s)^j||\hat{w}||^{2j}_{L^\infty_t W^{k,\infty}_x}\\
   & \lesssim \sum^{k}_{j=1}\ln(s)^jC^{2j}.
\end{align*}
Finally, up to a change of $1>\delta>0$, 
\begin{align}
||\Theta\mathfrak{R}^n_r(s)||_{W^{q,+\infty}_x}\lesssim s^{-n-r+1-\delta}h(C)C^{2n+1}
\end{align}
for $s\geq 1$ and some polynomial $h$. 
\end{proof}

\section{Local existence result}
\label{section:localexist}
In this section, we give a local existence result for \eqref{eq:NLS} in the proper norm. \\
\begin{unTheorem}
\label{unTheorem:localexist}
Let $\varepsilon>0$ be sufficiently small and let $(\alpha_j)_{0\leq j\leq 2N+1}$ be an increasing positive sequence such that $0<6\alpha_{j-1}\leq\alpha_j$ and $\alpha_{2n+1}<1/(16d+8)$. Then, there exists a finite $T>0$, such that for any $T_\star\geq 1$ and for any initial data $u_\star$ such that $||u_\star||_{H^{1,0}_x}+\sum^{2N+1}_{j=0}||T_\star^{-\alpha_j}e^{-iT_\star\Delta/2}u_\star||_{H^{0,j}_x}<\varepsilon$ there exists a unique local solution $u\in C(I,H^{1,0}_x)$ to \eqref{eq:NLS}, for $I=[T_\star,T_\star,+T]$, satisfying $u|_{t=T_\star}=u_\star$. Moreover, the solution satisfies 
\begin{align}
\label{eq:initialboundbootstrap}
&||u||_{C^0(I,H^{1,0}_x)}+\sum^{2N+1}_{j=0}||t^{-\alpha_j}f||_{C^0(I,H^{0,j}_x)}<C_X(\varepsilon)
\end{align}
where $f(t,x)=e^{-it\Delta/2}u(t,x)$ and $\lim_{\sigma\to0}C_X(\sigma)=0$.
\end{unTheorem}
For the purpose of our asymptotic analysis, it is crucial that the profile of our solution, $f(t)=e^{-it\Delta/2}u(t)$, remains bounded in weighted Sobolev norms, as in \eqref{eq:initialboundbootstrap}. A natural approach would be to rely on results of the type of \cite{zbMATH04115164}, to propagate bounds on weighted norms on $u$, and then use the fact that
\begin{align}
e^{it\Delta/2}xe^{-it\Delta/2}=x-it\partial_x
\end{align}
that yields
\begin{align}
\label{eq:commutJ}
[e^{it\Delta/2},x]=-it\partial_xe^{it\Delta/2},
\end{align}
to convert weights on $u$ into weights on the profile $f$. More precisely, the commutation relation \eqref{eq:commutJ} implies that, for any integer $m$, the quantity $x^mf(t,x)$ can be expressed as a linear combination of $e^{-it\Delta/2}x^{m-l}\partial^l_xu$, for $l\leq m$, together with lower–order derivative/weight contributions, using a non commutative binomial formula\footnote{We can obtain it from the Hausdorff-Campbell-Baker formula or by using \cite{arXiv:1707.03861}. In particular, the algebra here corresponds to the Heisenberg algebra.}, that is, 
\begin{align}
\label{eq:noncommut}
\nonumber x^mf(t,x)&=e^{-it\Delta/2}(-it\partial_x+x)^mu(t,x)\\
&=e^{-it\Delta/2}\sum_{j+k+2l=m}\frac{n!}{j!k!l!}(-\frac{it}{2})^lx^j(-it\partial_x)^ku.
\end{align}
However, this strategy requires that $u$, and so its initial data $u_\star$, belongs to $H^{2N+1,0}_x$. This regularity is mandatory to propagate the weighted norms of $u$, as explained in \cite{zbMATH04115164}, and subsequently to control the weighted norms of $f$ via the formula \ref{eq:noncommut}. To avoid the requirement of more regular initial data than $H^1_x$, we directly propagate weights on the profile $f$. \\\\
The idea here is to modify the standard fixed point argument to include the preservation of the proper weighted norms in each iteration, and then pass this property to the limit. We first recall that a standard result from \cite{zbMATH02001571}. \\
\begin{lem}[Lemma from \cite{zbMATH02001571}]
\label{lem:cazenave}
Let $E=\{u\in L^\infty_t(I,H^1_x):||u||_{L^\infty_t(I,H^1_x)}\leq M\}$. Then, the map 
\begin{align}
\label{eq:defmap}
\mathcal{H}(u)(t)=e^{i(t-T_\star)\Delta/2}u_\star-i\int^t_{T_\star}e^{i(t-s)\Delta/2}\left(\sum_{n=1}^{d}\lambda_n|u|^{2n}u\right)ds
\end{align}
is a contracting map from $E$ to $E$ for $T>0$, $M>0$ and $\varepsilon>0$ sufficiently small. 
\end{lem}
Now, we show that this map preserves the weighted norm.\\
\begin{lem}
\label{lem:preserve}
Define the norm $||u||_A=\sum^{2N+1}_{j=0}||t^{-\alpha_j}e^{-it\Delta/2}u||_{L^\infty_t(I,H^{0,j}_x)}$. Let $E_A=E\cap\{u\in L^\infty_t(I,H^1_x):||u||_A\leq M\}$. Then, for $T>0$, $M>0$ and $\varepsilon>0$ sufficiently small, the map 
$\mathcal{H}$ preserves $E_A$. 
\end{lem}
\begin{proof}
By \eqref{eq:defmap} and Lemma \ref{lem:phasenoxi}, the Fourier transform of the profile $\mathcal{I}(u)=e^{-it\Delta/2}\mathcal{H}(u)$ solves 
\begin{align*}
\widehat{\mathcal{I}(u)}(t,\xi)=\hat{f_\star}(\xi)-i\sum_{n=1}^{d}\frac{\lambda_n}{(2\pi)^{n}}\int^t_{T_\star}\int e^{is\Psi^n(\eta)}\mathcal{N}^n(s,\xi,\eta)d\eta ds.
\end{align*}
where $f_\star=e^{-iT_\star\Delta/2}u_\star$. Differentiating with respect to $\xi$ leads to 
\begin{align*}
\partial_\xi^j\widehat{\mathcal{I}(u)}(t,\xi)&=\partial_\xi^j\hat{f_\star}(\xi)-i\sum_{n=1}^{d}\frac{\lambda_n}{(2\pi)^{n}}\int^t_{T_\star}\int e^{is\Psi^n(\eta)}\partial_\xi^j\mathcal{N}^n(s,\xi,\eta)d\eta ds,
\end{align*}
for $j\in\mathbb{N}$, $0\leq j\leq 2n+1$, and for 
\begin{align*}
\partial_\xi^j\mathcal{N}^n=\sum_{\overrightarrow{m}\in\mathbb{N}^{2n+1},|\overrightarrow{m}|=j}^{2n+1}\binom{j}{\overrightarrow{m}}\left(\prod^{n}_{k=1}\partial_\xi^{\overrightarrow{m}_{2k+1}}\hat{f}(H^n_{k+1}-\eta_{2k})\partial_\xi^{\overrightarrow{m}_{2k}}\overline{\hat{f}(H^n_{k}-\eta_{2k})}\right)\partial_\xi^{\overrightarrow{m}_1}\hat{f}(H^n_1),
\end{align*}
where $f(s)=e^{-is\Delta/2}u(s)$. Then, applying the inverse change of variable in $\eta$ of that of lemma \ref{lem:phasenoxi} to recover the original unknowns of notation \ref{nota:Pfraknota} leads to
\begin{align*}
\frac{1}{(2\pi)^{n}}\int e^{is\Psi^n(\eta)}\partial_\xi\mathcal{N}^n(s,\xi,\eta)d\eta =\sum_{\overrightarrow{m}\in\mathbb{N}^{2n+1},|\overrightarrow{m}|=j}^{2n+1}\binom{j}{\overrightarrow{m}}\frac{1}{(2\pi)^n}\int e^{is\Phi^n(\xi,\eta)}\mathcal{X}^n_{\overrightarrow{m}}(s,\xi,\eta)d\eta, 
\end{align*}
with 
\begin{align}
\mathcal{X}^n_{\overrightarrow{m}}(s,\xi,\eta)= 
\partial_\xi^{\overrightarrow{m}_{2n+1}}\hat{f}(\xi-\eta_{2n})\partial_\xi^{\overrightarrow{m}_{2n}}\overline{\hat{f}(\eta_{2n-1}-\eta_{2n})}\left(\prod^{n-1}_{k=1}\partial_\xi^{\overrightarrow{m}_{2k+1}}\hat{f}(\eta_{2i+1}-\eta_{2k})\partial_\xi^{\overrightarrow{m}_{2k}}\overline{\hat{f}(\eta_{2k-1}-\eta_{2k})}\right)\partial_\xi^{\overrightarrow{m}_1}\hat{f}(\eta_1).
\end{align}
Let $\sup_{k\leq 2n+1}\overrightarrow{m}_{k}=\overrightarrow{m}_{2n+1}$ without loss of generality. Distributing the phase $\Phi^n(\xi,\eta)$ back onto each profile so as to reconstruct the Fourier transform of the Schrödinger propagator acting on every $\hat{f}$ and using Plancherel identity yields
\begin{align*}
||\frac{1}{(2\pi)^n}\int e^{is\Phi^n(\xi,\eta)}\mathcal{X}^n_{\overrightarrow{m}}(s,\xi,\eta)d\eta||_{L^2_\xi}&\leq ||e^{is\Delta/2}(x^{\overrightarrow{m}_{2n+1}}f(s))\left(\prod^{n}_{k=1}\overline{e^{is\Delta/2}(x^{\overrightarrow{m}_{2k}}f(s))}e^{is\Delta/2}(x^{\overrightarrow{m}_{2k-1}}f(s))\right)||_{L^2_x}\\
&\leq ||x^{\overrightarrow{m}_{2n+1}}f(s)||_{L^2_x}\prod^{2n}_{k=1}||e^{is\Delta/2}(x^{\overrightarrow{m}_{k}}f(s))||_{L^\infty_x}\\
&\lesssim s^{-n}||\langle x\rangle^{\overrightarrow{m}_{2n+1}}f(s)||_{L^2_x}\prod^{2n}_{k=1}||x^{\overrightarrow{m}_{k}}f(s)||_{L^1_x}\\
&\lesssim s^{-n}||\langle x\rangle^{j}f(s)||_{L^2_x}^{2n+1}.
\end{align*}
Thus, 
\begin{align*}
\int^t_{T_\star}\int e^{is\Psi^n(\eta)}\partial_\xi^j\mathcal{N}^n(s,\xi,\eta)d\eta ds&\lesssim \int^t_{T_\star}s^{-n}||\langle x\rangle^{j}f(s)||_{L^2_x}^{2n+1}ds\\
&\lesssim \int^t_{T_\star}s^{(2n+1)\alpha_j-n }||s^{-\alpha_j}\langle x\rangle^{j}f(s)||_{L^2_x}^{2n+1}ds\\
&\lesssim (t^{(2n+1)\alpha_j-n+1}-T_\star^{(2n+1)\alpha_j-n+1})M^{2n+1},
\end{align*}
which implies  
\begin{align*}
||x^j\mathcal{I}(u)(t)||_{L^2_x}&=||\partial_\xi^j\mathcal{I}(u)(t)||_{L^2_\xi}\\
&\leq ||\partial_\xi^j\hat{f_\star}(\xi)||_{L^2_\xi}+\sum_{n=1}^{d}\frac{\lambda_n}{(2\pi)^{n}}||\int^t_{T_\star}\int e^{is\Psi^n(\eta)}\partial^j_\xi\mathcal{N}^n(s,\xi,\eta)d\eta ds||_{L^2_\xi}\\
&\leq T_\star^{\alpha_j}\varepsilon+C(t^{(2n+1)\alpha_j-n+1}-T_\star^{(2n+1)\alpha_j-n+1})M^{2n+1}
\end{align*}
for a certain $C>0$ that only depends on universal constants and the coefficients of the polynomial nonlinearity. Finally, for $T>0$, $M>0$ and $\varepsilon>0$ sufficiently small,  we obtain 
\begin{align}
||\mathcal{H}(u)||_A\leq ||u||_A.
\end{align}
\end{proof}

\begin{proof}[Proof of Theorem \ref{unTheorem:localexist}]
Using lemma \ref{lem:cazenave} together with standard fixed point argument, we deduce that, for $T>0$ and $\varepsilon>0$ sufficiently small, there exists a sequence $(u_k)_{k\in\mathbb{N}}$ defined by $u_{k+1}=\mathcal{H}(u_k)$ which converges in $L^\infty_t(I,H^1_x)$ to $u\in L^\infty_t(I,H^1_x)$ the unique solution to \eqref{eq:NLS} with initial data $u_\star$. Then, since the Schrödinger propagator $e^{-it\Delta/2}$ is an isometry on $H^1$, the sequence of profiles $(f_k)_{k\in\mathbb{N}}=e^{-it\Delta/2}u_k$ converges in $L^\infty_t(I,H^1_x)$ to $f=e^{-it\Delta/2}u$ which is solution to \eqref{eq:NLSf}. Moreover, by lemma \ref{lem:preserve}, the following uniform bound holds: 
\begin{align}
\sup_{k\in\mathbb{N}}||u_k||_A\leq C(\varepsilon,T)
\end{align} 
or equivalently
\begin{align}
\sup_{k\in\mathbb{N}}(\sum_{j=0}^{2n+1}||t^{-\alpha_j}f_k||_{L^{\infty}_t(I,H^{0,j}_x)})\leq C(\varepsilon,T).
\end{align} 
To pass to the limit in these weighted norms, we proceed as follows. Identifying $\langle x\rangle^jdx$ with a measure $d\mu^j$, we may regard $H^{0,j}_x$ as $B=L^2(\mathbb{R},d\mu^j)$. The latter is reflexive and thus enjoys the Radon-Nikodym property, see \cite{zbMATH03576139}, Corollary 13, or \cite{zbMATH06644896}. Then, identifying further $t^jdt$ with a measure $d\nu^j$, we obtain uniform bounds on $L^{\infty}(I,B,d\nu^j)$, which is the dual of $L^1(I,B^\star,d\nu^j)$, where $B^\star$ denotes the dual space of $B$. Thus, by Banach-Alaoglu Theorem, there exists a subsequence $(f_{\sigma(k)})_{k\in\mathbb{N}}$ that converges in the weak-$\star$ topology of each space $L^\infty(I,B,d\nu^j)$. By uniqueness, the weak limits must coincide with the strong limit in $L^\infty_t(I,H^1_x)$. Consequently, we obtain 
\begin{align}
\sum_{j=0}^{2n+1}||t^{-\alpha_j}f||_{L^{\infty}_t(I,H^{0,j}_x)}\leq C(\varepsilon,T).
\end{align} 
The continuity in time follows from the evolution equations \eqref{eq:NLS} and \eqref{eq:NLSf}, which allows us to recover the statement \eqref{eq:initialboundbootstrap}.
\end{proof}
\begin{cor}[Corollary of Theorem \ref{unTheorem:localexist}]
\label{cor:modprof}
Let $u\in C(I,H^{1,0}_x)$ be the solution to \eqref{eq:NLS} given by Theorem \ref{unTheorem:localexist} with initial data $||u_\star||_{H^{1,0}_x}+\sum^{2N+1}_{j=0}||T_\star^{-\alpha_j}e^{-iT_\star\Delta/2}u_\star||_{H^{0,j}_x}<\varepsilon$ for some $\varepsilon>0$. Let $f(t,x)=e^{-it\Delta/2}u(t,x)$ be the associated profile and let $\hat{\omega}(t,x)=e^{i\int^t_{T_\star}\frac{\lambda_1|\hat{f}(s,\xi)|^2}{s}ds}f(t,x)$. Then,
\begin{align}
\label{eq:omegaevoineq}
||\hat{\omega}||_{C^0(I,W^{2N,\infty}_\xi)}\leq C_X(\varepsilon),
\end{align}
where $C_X$ is the same as in Theorem \ref{unTheorem:localexist}, without loss of generality.\\
\end{cor}
\begin{rem}
\label{rem:usecor}
We observe that, when $T_\star=1$, the term $\hat{\omega}$ is equal to the Fourier transform of the modified profile $\hat{w}$. 
In view of the bootstrap argument developed in the next section, Corollary \ref{cor:modprof} and Theorem \ref{unTheorem:localexist} are typically applied in order to continue the solution beyond a time $T_\star>1$. In this context, one has $\hat{\omega}=e^{i\int^{T_\star}_{1}\frac{\lambda_1|\hat{w}(s,\xi)|^2}{s}ds}\hat{w}$, so that, if $\hat{w}$ satisfies \eqref{eq:omegaevoineq} on the interval $[1,T_\star]$, then Corollary \ref{cor:modprof} ensures that $\hat{w}$ also satisfies \eqref{eq:omegaevoineq} on $[1,T_\star+T]$.
\end{rem}
\begin{proof}
From \eqref{eq:showfirstorder}, direct computations show that $\hat{\omega}=$ is solution to  
\begin{equation}
\label{eq:omegaevo}
    \partial_t\hat{\omega}(t,\xi)=-i\theta(t,\xi)\mathfrak{R}^1_1(t,\xi)-i\theta(t,\xi)\mathfrak{R}^{>1}_0(t,\xi)
\end{equation}
where $\mathfrak{R}^1_1$ and $\mathfrak{R}^{>1}_0$ are defined with respect to $f$ in Lemma \ref{lem:exactformP} and Notation \ref{nota:sumtoterror}, respectively, and where $\theta=e^{i\int^t_{T_\star}\frac{\lambda_1|\hat{f}(s,\xi)|^2}{s}ds}$. Integrating \eqref{eq:omegaevo} and differentiating it $j$ times with respect to $\xi$ yields
\begin{equation}
\label{eq:modifprofilelocal}
  \partial_\xi^j\hat{\omega}(t,\xi)= \partial_\xi^j\hat{\omega}(T_\star,\xi)-i\int^t_{T_\star}\sum_{\overrightarrow{m}\in\mathbb{N}^{2}, |\overrightarrow{m}|=j}\binom{j}{\overrightarrow{m}}\left(\partial_\xi^{\overrightarrow{m}_1}\theta(t,\xi)\partial_\xi^{\overrightarrow{m}_2}\mathfrak{R}^1_{1}(s,\xi)+\partial_\xi^{\overrightarrow{m}_1}\theta(t,\xi)\partial_\xi^{\overrightarrow{m}_2}\mathfrak{R}^{>1}_0(s,\xi)\right)ds.
\end{equation}
Due to the control on the weighted norms of $f$ given in \eqref{eq:initialboundbootstrap}, we have
\begin{align*}
||\hat{f}(t)||_{W^{2N,\infty}_\xi}\leq t^{\alpha_{2N+1}}C_X(\varepsilon)
\end{align*} 
by Sobolev embedding, which implies that 
\begin{align*}
||\theta(t)||_{W^{2N,\infty}_\xi}\leq t^{\beta}C(\varepsilon),
\end{align*}
where $\lim_{\sigma\to0}C(\sigma)=0$ and where $\beta>0$. Furthermore, by proposition \ref{propal:controlerror} on the interval $I$ and using again the bounds provided by \eqref{eq:initialboundbootstrap}, we obtain 
\begin{align*}
||\mathfrak{R}^1_{1}(t)||_{W^{2N,\infty}_\xi}\leq t^{-\gamma}C(\varepsilon)
\end{align*} 
and 
\begin{align*}
||\mathfrak{R}^{>1}_0(t)||_{W^{2N,\infty}_\xi}\leq t^{-\gamma}C(\varepsilon),
\end{align*} 
for some $\gamma>0$ and for the same $C$, without loss of generality. Combining these estimates, \eqref{eq:modifprofilelocal} implies  \eqref{eq:omegaevoineq} straightforwardly, which concludes the proof.
\end{proof}
\begin{cor}[Corollary of Theorem \ref{unTheorem:localexist}]
\label{cor:decay}
Let $u\in C(I,H^{1,0}_x)$ be the solution to \eqref{eq:NLS} given by Theorem \ref{unTheorem:localexist} associated with initial data satisfying $||u_\star||_{H^{1,0}_x}+\sum^{2N+1}_{j=0}||T_\star^{-\alpha_j}e^{-iT_\star\Delta/2}u_\star||_{H^{0,j}_x}<\varepsilon$ for some $\varepsilon>0$. Let $f(t,x)=e^{-it\Delta/2}u(t,x)$ denotes the associated profile. Then, the following decay estimate holds
\begin{align}
\label{eq:sharpdecay}
||t^{1/2}u||_{C^0(I,L^\infty_x)}\leq C_X(\varepsilon),
\end{align}
where $C_X$ is the same constant as in Theorem \ref{unTheorem:localexist} and Corollary \ref{cor:modprof}, without loss of generality.
\end{cor}
We first recall a classical estimate for the Schrödinger semigroup, taken from \cite{zbMATH01192427} (also used in \cite{zbMATH06033880}). \\
\begin{lem}[Lemma from \cite{zbMATH01192427}]
\label{lem:semigroup}
The following bound holds: 
\begin{align}
||e^{it\Delta/2}\phi||_{L^\infty_x}\lesssim t^{-1/2}||\hat{\phi}||_{L^\infty_\xi}+t^{-1/2-\delta}||\phi||_{H^{0,\gamma}_x}
\end{align}
for $\gamma>1/2+2\delta$. 
\end{lem}
\begin{proof}[Proof of Corollary \ref{cor:decay}]
By Corollary \ref{cor:modprof} together with remark \ref{rem:usecor}, we have $|\hat{f}|=|\hat{\omega}|$ and $||\hat{\omega}||_{C^0(I,L^{\infty}_\xi)}\leq C_X(\varepsilon)$. Then, the assumption on the sequence $(\alpha_j)_{0\leq j\leq 2N+1}$ in Theorem \ref{unTheorem:localexist} ensures the existence of $\delta\in\mathbb{R}$ such that $\alpha_1<\delta<1/4$. Finally, lemma \ref{lem:semigroup} implies that
\begin{align*}
t^{1/2}||u(t)||_{L^\infty_x}&=t^{1/2}||e^{it\Delta/2}f(t)||_{L^\infty_x}\\
&\lesssim ||\hat{f}(t)||_{L^\infty_\xi}+t^{-\delta}||f(t)||_{H^{0,1}_x}\\
&\lesssim ||\hat{\omega}(t)||_{L^\infty_\xi}+t^{-\delta+\alpha_1}||t^{-\alpha_1}f(t)||_{H^{0,1}_x}.
\end{align*} 
Then, the bound \ref{eq:omegaevoineq} of Corollary \ref{cor:modprof} and the bound \eqref{eq:initialboundbootstrap} of Theorem \ref{unTheorem:localexist} yield directly \eqref{eq:sharpdecay}. 
\end{proof}
\section{Global existence and bootstrap}
\label{section:bootstrap}
In this section, we show that the solution to \eqref{eq:NLS} constructed in Theorem \ref{unTheorem:localexist} extends globally and satisfies uniform bounds with respect to the bootstrap norm $||u||_{X_T}$ defined below.\\
\begin{nota}
\label{nota:bootstrapnorm}
We define the bootstrap norm by 
\begin{align}
||u||_{X_T}:=||u||_{C^0_TH^{1,0}_x}+\sum^{2N+1}_{j=0}||t^{-\alpha_j}f||_{C^0_TH^{0,j}_x}+||\hat{w}||_{C^0_TW^{2N,\infty}_\xi}+||t^{1/2}u||_{C^0_TL^\infty_x}
\end{align}
for $0<6\alpha_{j-1}\leq\alpha_j$ with $\alpha_{2n+1}<1/(16d+8)$ and some $T>0$.
\end{nota}
Firstly, we state the bootstrap argument. \\
\begin{propal}
\label{propal:bootstrap}
Assume that $u$ is a local solution to \eqref{eq:NLS} on an interval $[1,T]$ satisfying
\begin{align}
||u||_{X_T}<\varepsilon_1,
\end{align}
for some $\varepsilon_1>0$, and whose initial data $u|_{t=1}=u_1$ satisfies 
\begin{align}
||u_1||_{H^{1,0}_x}+\sum^{2N+1}_{j=0}||e^{-i\Delta/2}u_1||_{H^{0,j}_x}<\varepsilon_0
\end{align}
for some $\varepsilon_0>0$. Then, for $\varepsilon_0>0$ sufficiently small
\begin{align}
||u||_{X_T}\lesssim \varepsilon_0+\Gamma(\varepsilon_1)
\end{align}
where $\Gamma$ denotes a polynomial whose first non-zero coefficient is at least of degree $3$.
\end{propal}
The proof is given in the following sequence of lemmas. \\
\begin{lem}
Under the hypotheses of proposition \ref{propal:bootstrap}, one has 
$||u||_{L^\infty_TH^{1,0}_x}\leq \varepsilon_0 +\Gamma(\varepsilon_1)$,
for $\Gamma$ a polynomial whose first non-zero coefficient is at least of degree $3$. 
\end{lem}
\begin{proof}
We use a standard argument from \cite{zbMATH02001571}. Because of the specific structure of the polynomial nonlinearity $\sum_{n=1}^{d}\lambda_n|u|^{2n}u$, the energy associated with \eqref{eq:NLS} is simply
\begin{align}
\mathcal{E}(t)=\int_\mathbb{R}\frac{|\partial_xu(t)|^2}{2}+\sum_{n=1}^{d}\lambda_n\frac{|u(t)|^{2n+2}}{2n+2}dx.
\end{align}
The quantity $\mathcal{E}(t)$ is constant, as well is the total mass $||u(t)||_{L^2_x}$. Combining these two conservation laws yields control of the $H^1$ norm. More precisely, 
\begin{align*}
||u(t)||^2_{H^1_x}&\leq \mathcal{E}(t)+||u(t)||^2_{L^2_x}-\int_\mathbb{R}\sum_{n=1}^{d}\lambda_n\frac{|u(t)|^{2n+2}}{2n+2}dx\\
&\leq \mathcal{E}(0)+||u(0)||^2_{L^2_x}-\sum_{n=1}^{d}\lambda_n||u||_{H^1_x}^{2n}\int_\mathbb{R}\frac{|u(t)|^{2}}{2n+2}dx\\
&\leq h_1(\varepsilon_0)+\varepsilon_0^2+h_2(\varepsilon_1)||u(t)||_{H^1_x}^{2}\varepsilon_0^2
\end{align*}
where $h_1$ is a polynomial whose first non-zero coefficient is at least of degree $2$ and $h_2$ is a generic polynomial. For a fixed $\varepsilon_1$ and for $\varepsilon_0$ sufficiently small, this yields 
\begin{align*}
||u||_{C^0_TH^1_x}\leq h_3(\varepsilon_0)
\end{align*}
for some polynomial $h_3$ whose first non-zero coefficient is at least of degree $1$. This concludes the proof.\\
\end{proof}
\begin{lem}
\label{lem:bootstrapflem}
Under the hypotheses of proposition \ref{propal:bootstrap}, one has
$\sum^{2N+1}_{j=0}||t^{-\alpha_j}f||_{C^0_TH^{0,j}_x}\lesssim \varepsilon_0 +\Gamma(\varepsilon_1)$,
for $\Gamma$ a polynomial whose first non-zero coefficient is at least of degree $3$. \\
\end{lem}
\begin{rem}
Observe that a sequence $(\alpha_j)_{0\leq j\leq 2N+1}$ is required and not merely a single general $\alpha$. The reason for this hierarchy becomes clear in the proof that follows.
\end{rem}
\begin{proof}
We want to improve the bound for $||t^{-\alpha_j}f||_{C^0_T H^{0,j}_x}$ for a general $0\leq j \leq 2N+1$. For that, we use a strategy that recalls that of Theorem \ref{unTheorem:localexist} for the local existence and the boundedness of weighted norms, although the argument here is more precise. We start with the evolution equation satisfied by $\hat{f}$
\begin{equation}
\label{eq:recallNLSduhamfFourier}
    \hat{f}(t,\xi)=\hat{u}_\star(\xi)-i\hat{P}(u),
\end{equation}
with 
\begin{equation}
    \hat{P}(u)(t)=\sum_{n=1}^{d}\lambda_n\hat{P}^n(u)(t),
\end{equation}
and differentiate it. Using notation \ref{nota:Pfraknota} and Lemma \ref{lem:phasenoxi}, we have 
\begin{align*}
\partial^{j}_\xi\hat{f}(t,\xi)&=\partial^{j}_\xi\hat{u}_\star(\xi)-i\partial^{j}_\xi\hat{P}(\hat{f})(t,\xi)\\
&=\partial^{j}_\xi\hat{u}_\star(\xi)-i\sum_{n=1}^{d}\frac{\lambda_n}{(2\pi)^{n}}\int^t_1\int e^{is\Psi^n(\eta)}\partial^{j}_\xi\mathcal{N}^n(s,\xi,\eta)d\eta_1\dots d\eta_{2n}ds,
\end{align*}
and with notation \ref{nota:Nonlinearity}, we have 
\begin{align*}
\partial^{j}_\xi\mathcal{N}^n=\sum_{\overrightarrow{m}\in\mathbb{N}^{2n+1},|\overrightarrow{m}|=j}\binom{j}{\overrightarrow{m}}\prod^{n}_{k=1}\partial_\xi^{\overrightarrow{m}_{2k+1}}\hat{f}(H^n_{k+1}-\eta_{2k})\partial_\xi^{\overrightarrow{m}_{2k}}\overline{\hat{f}(H^n_{k}-\eta_{2i})}\partial_\xi^{\overrightarrow{m}_{1}}\hat{f}(H^n_1).
\end{align*}
After performing the change of variables in $\eta$ to recover the original unknowns of Notation \ref{nota:Pfraknota}, we obtain
\begin{align*}
\int^t_1\int e^{is\Psi^n(\eta)}\partial^{j}_\xi\mathcal{N}^n(s,\xi,\eta)d\eta_1\dots d\eta_{2n}ds&=\int^t_1\sum_{\overrightarrow{m}\in\mathbb{N}^{2n+1}, |\overrightarrow{m}|=j}\binom{j}{\overrightarrow{m}}\int e^{is\Phi^n(\xi,\eta)}\mathcal{X}^{n}_{\overrightarrow{m}}(s,\xi,\eta)d\eta_1\dots d\eta_{2n}ds
\end{align*}
with 
\begin{align}
\mathcal{X}^{n}_{\overrightarrow{m}}(\xi,\eta)=\partial^{\overrightarrow{m}_{2n+1}}\hat{f}(\xi-\eta_{2n})\partial^{\overrightarrow{m}_{2n}}\overline{\hat{f}(\eta_{2n-1}-\eta_{2n})}\prod^{n-1}_{i=1}\partial^{\overrightarrow{m}_{2i+1}}\hat{f}(\eta_{2i+1}-\eta_{2i})\partial^{\overrightarrow{m}_{2i}}\overline{\hat{f}(\eta_{2i-1}-\eta_{2i})}\partial^{\overrightarrow{m}_{1}}\hat{f}(\eta_1).
\end{align}
Then, distributing the phase so as to recover the Schrödinger propagator acting on each $\hat{f}$ in the product, and using Plancherel's identity, leads to
\begin{align*}
||\frac{1}{(2\pi)^{n}}\int^t_1\int e^{is\Psi^n(\eta)}\partial^{j}_\xi\mathcal{N}^n(s,\xi,\eta)d\eta_1\dots d\eta_{2n}ds||_{L^2_\xi}&=||\int^t_1\sum_{\overrightarrow{m}\in\mathbb{N}^{2n+1}, |\overrightarrow{m}|=j}\binom{j}{\overrightarrow{m}}e^{-is\partial^2_{xx}}\mathcal{Y}^{n}_{\overrightarrow{m}}(s,x)ds||_{L^2_x}
\end{align*}
with 
\begin{align}
\mathcal{Y}^{n}_{\overrightarrow{m}}(s,x)=\prod^{n}_{i=1}e^{is\partial^2_{xx}}(x^{\overrightarrow{m}_{2i+1}}f(s,x))\overline{e^{is\partial^2_{xx}}(x^{\overrightarrow{m}_{2i}}f(s,x)})e^{is\partial^2_{xx}}(x^{\overrightarrow{m}_{1}}f(s,x)).
\end{align}
Observe that if all derivatives fall on one $\hat{f}$ (or equivalently if all weights in $x$ are on one $f$), that is, if $\overrightarrow{m}:=(j,\dots,0)$, $ \overrightarrow{m}:=(0,\dots,j)$ or $\overrightarrow{m}:=(0,\dots,j,\dots,0)$, then 
\begin{align*}
||\int^t_1e^{-is\partial^2_{xx}}\mathcal{Y}^{n}_{\overrightarrow{\mu}}(s)ds||_{L^2_x}&\leq \int^t_1||\langle x\rangle^{j}f(s)||_{L^2_x}||u||^{2n}_{L^\infty_x}ds\\
&\leq \int^t_1t^{\alpha_j}||t^{-\alpha_j}\langle x\rangle^{j}f(s)||_{L^2_x}\varepsilon^{2n}_1t^{-n}ds\\
&\leq t^{\alpha_j-n+1}\varepsilon^{2n+1}_1.
\end{align*}
If only this contribution were present, the bootstrap improvement would hold without any hierarchy among the exponents $(\alpha_j)_{0\leq j\leq 2N+1}$. Moreover, one sees that the limiting case preventing $\alpha_j$ to be $0$ is $n=1$. \\\\
For a general $\overrightarrow{m}\in\mathbb{N}^{2n+1}$, we set $\overrightarrow{m}_1=\max_{0\leq j\leq 2N+1}\overrightarrow{m}_j$ without loss of generality, and compute
\begin{align*}
||\int^t_1e^{-is\partial^2_{xx}}\mathcal{Y}^{n}_{\overrightarrow{m}}(s)ds||_{L^2_x}&\leq \int^t_1||\langle x\rangle^{\overrightarrow{m}_{1}}f(s)||_{L^2_x}\prod^{n}_{i=1}||e^{is\partial^2_{xx}}(x^{\overrightarrow{m}_{2i+1}}f(s))||_{L^\infty_x}||\overline{e^{is\partial^2_{xx}}(x^{\overrightarrow{m}_{2i}}f(s,x)})||_{L^\infty_x}ds\\
&\leq \int^t_1t^{-n}||\langle x\rangle^{\overrightarrow{m}_{1}}f(s)||_{L^2_x}\prod^{n}_{i=1}||\langle x\rangle^{\overrightarrow{m}_{2i+1}}f(s)||_{L^1_x}||\langle x\rangle^{\overrightarrow{m}_{2i}}f(s))||_{L^1_x}ds\\
&\leq \int^t_1t^{-n}||\langle x\rangle^{\overrightarrow{m}_{1}}f(s)||_{L^2_x}\prod^{n}_{i=1}||\langle x\rangle^{\overrightarrow{m}_{2i+1}+\delta+1/2}f(s)||_{L^2_x}||\langle x\rangle^{\overrightarrow{m}_{2i}+\delta+1/2}f(s))||_{L^2_x}ds.
\end{align*}
For $n>1$,l each weighted norm can be controlled by the maximal weight exponent $\alpha_j$, paying for $t^{-\alpha_j}$, which yields
\begin{align*}
||\int^t_1e^{-is\partial^2_{xx}}\mathcal{Y}^{n}_{\overrightarrow{m}}(s)ds||_{L^2_x}&\leq \int^t_1t^{-n+(2n+1)\alpha_j}||t^{-\alpha_j}\langle x\rangle^{j}f(s)||_{L^2_x}\prod^{n}_{i=1}||t^{-\alpha_j}\langle x\rangle^{j}f(s)||_{L^2_x}||t^{-\alpha_j}\langle x\rangle^{j}f(s))||_{L^2_x}ds\\
&\lesssim t^{-n+1+(2n+1)\alpha_j}\varepsilon^{2n+1}_1\\
&\lesssim t^{\alpha_j}\varepsilon^{2n+1}_1
\end{align*}
where we use the fact that $-n+1+2n\alpha_j<0$, which follows from the restriction $\alpha_j<1/4$ from the definition of our sequence $(\alpha_j)_{0\leq j\leq 2N+1}$.\\\\
For $n=1$, we must rely on interpolations\footnote{Without the interpolation, the case $j=2$ does not work for $\overrightarrow{m}_1=\overrightarrow{m}_2=1$ and $\overrightarrow{m}_3=0$ for example.} between weighted norms given by lemma \ref{lem:interpol}. We have
\begin{align*}
||\int^t_1e^{-is\partial^2_{xx}}\mathcal{Y}^{n}_{\overrightarrow{m}}(s)ds||_{L^2_x}&\leq \int^t_1t^{-1}||\langle x\rangle^{\overrightarrow{m}_{1}}f(s)||_{L^2_x}||\langle x\rangle^{\overrightarrow{m}_{2}+\delta+1/2}f(s)||_{L^2_x}||\langle x\rangle^{\overrightarrow{m}_{3}+\delta+1/2}f(s))||_{L^2_x}ds\\
&\lesssim \int^t_1t^{-1+\alpha_{\overrightarrow{m}_1}}\varepsilon_{1}\left(||\langle x\rangle^{\overrightarrow{m}_{2}}f(s)||_{L^2_x}^{1/2-\delta}||\langle x\rangle^{\overrightarrow{m}_{2}+1}f(s)||_{L^2_x}^{1/2+\delta}\right)\\
&\left(||\langle x\rangle^{\overrightarrow{m}_{3}}f(s))||_{L^2_x}^{1/2-\delta}||\langle x\rangle^{\overrightarrow{m}_{3}+1}f(s))||_{L^2_x}^{1/2+\delta}\right)ds\\
&\lesssim \int^t_1t^{-1+\alpha_{\overrightarrow{m}_1}+\alpha_{\overrightarrow{m}_2}(1/2-\delta)+\alpha_{\overrightarrow{m}_2+1}(1/2+\delta)+\alpha_{\overrightarrow{m}_3}(1/2-\delta)+\alpha_{\overrightarrow{m}_3+1}(1/2+\delta)}\varepsilon_{1}^3ds\\
&\lesssim t^{\alpha_j}
\end{align*}
where we use the fact that $\alpha_{\overrightarrow{m}_1}+\alpha_{\overrightarrow{m}_2}(1/2-\delta)+\alpha_{\overrightarrow{m}_2+1}(1/2+\delta)+\alpha_{\overrightarrow{m}_3}(1/2-\delta)+\alpha_{\overrightarrow{m}_3+1}(1/2+\delta)\leq \alpha_j$, which follows from the restriction $\alpha_j\geq 6\alpha_{j-1}$ from the definition of our sequence $(\alpha_j)_{0\leq j\leq 2N+1}$.\\\\
We also remark that for $\lambda_1=0$, i.e., in the absence of cubic terms in \eqref{eq:NLS}, one can show that there is no growth in time for the Sobolev norms of $\hat{f}(t)$. \\\\
Finally, collecting all contributions yields
\begin{align*}
||\partial^{j}_\xi\hat{f}(t)||_{L^2_\xi}\lesssim \varepsilon_0+t^{\alpha_j}\Gamma(\varepsilon_1)
\end{align*}
and so
\begin{align*}
\sum^{2N+1}_{j=0}||t^{-\alpha_j}\partial^{j}_\xi\hat{f}||_{C^0_TL^2_\xi}\lesssim \varepsilon_0+\Gamma(\varepsilon_1)
\end{align*}
which is the desired inequality. 
\end{proof}
\begin{lem}
\label{lem:bootstrapwlem}
Under the hypotheses of proposition \ref{propal:bootstrap}, one has $||\hat{w}||_{C^0_TW^{2N,\infty}_\xi}\leq\varepsilon_0 +\Gamma(\varepsilon_1)$,
for $\Gamma$ a polynomial whose first non-zero coefficient is at least of degree $3$.
\end{lem}
\begin{proof}
The modified profile satisfies 
\begin{equation}
\label{eq:modifprofileforbootstrap}
  \hat{w}(t,\xi)=\hat{w}(1,\xi)-i\int^t_1\Theta(t,\xi)\mathfrak{R}^1_{1}(s,\xi)+\Theta(t,\xi)\mathfrak{R}^{>1}_0(s,\xi)ds.
\end{equation}
We recall proposition \ref{propal:controlerror}. Given $p\in\mathbb{N}$, if 
\begin{align*}
||t^{-\alpha}\langle x\rangle^{p+1} f||_{L^\infty_tL^2_x}+||\hat{w}||_{L^\infty_tW^{p,\infty}_\xi}\leq C
\end{align*}
for some $1/(16d+8)>\alpha>0$ and some $C>0$, then there exists $1>\delta>0$, a polynomial $h$, such that for $s\geq 1$, the error terms satisfy $||\Theta\mathfrak{R}^n_0(s)||_{W^{p,\infty}_\xi}\lesssim h(C)C^{2n+1}s^{-n+1-\delta}$, and for $r>0$, $||\Theta\mathfrak{R}^n_r(s)||_{W^{p+2-2r,\infty}_\xi}\lesssim h(C)C^{2n+1}s^{-n-r+1-\delta}$. In the right hand side of \eqref{eq:modifprofileforbootstrap} appears the term $\Theta\mathfrak{R}^1_1$ and a finite sum of $\Theta\mathfrak{R}^n_0$ terms for $n>1$ contained in $\mathfrak{R}^{>1}_0$, see notation \ref{nota:sumtoterror}. Then, using proposition \ref{propal:controlerror} cited above with $p=2N$, $C=\varepsilon_1$ and $\alpha=\alpha_{2N+1}$ leads to
\begin{align*}
  ||\hat{w}(t)||_{W^{2N,\infty}_\xi}\lesssim||\hat{w}(1,\xi)||_{W^{2N,\infty}_\xi}+\int^t_1\Gamma(\varepsilon_1)s^{-1-\delta}ds
\end{align*}
for some polynomial $\Gamma$ whose first non-zero coefficient is at least of degree $3$. This implies that
\begin{align*}
  ||\hat{w}||_{C^0_TW^{2N,\infty}_\xi}\lesssim \varepsilon_0+\Gamma(\varepsilon_1).
\end{align*}
\end{proof}
\begin{lem}
Under the hypotheses of proposition \ref{propal:bootstrap}, one has 
$||t^{1/2}u||_{C^0_TL^\infty_x}\lesssim \varepsilon_0 +\Gamma(\varepsilon_1)$,
for $\Gamma$ a polynomial whose first non-zero coefficient is at least of degree $3$. 
\end{lem}
\begin{proof}
From lemma \ref{lem:semigroup}, we know that
\begin{align*}
||t^{1/2}u||_{L^\infty_x}&=||t^{1/2}e^{it\Delta/2}f||_{L^\infty_x}\\
&\lesssim ||\hat{f}||_{L^\infty_\xi}+t^{-\delta}||f||_{H^{0,1}_x}\\
&\leq ||\hat{w}||_{L^\infty_\xi}+t^{\alpha_1-\delta}||t^{\alpha_1}f||_{H^{0,1}_x}
\end{align*}
for $1/4>\delta>\alpha_1$. Then, applying the lemmas \ref{lem:bootstrapflem} and \ref{lem:bootstrapwlem}, we get the desired result. 
\end{proof}
We are now ready to state the global existence result.
\begin{propal}
\label{propal:GWP}
For $\varepsilon_0>0$ sufficiently small, let $||u_1||_{H^{1,0}_x}+\sum^{2N+1}_{j=0}||e^{-i\Delta/2}u_1||_{H^{0,j}_x}<\varepsilon_0$. Then, there exists a global solution $u\in C([1,\infty],H^{1,0}_x)$ to \eqref{eq:NLS}. Moreover, there exists $C_0>$ that only depends on $\varepsilon_0$ such that
\begin{align}
\label{eq:GWPeq}
||u||_{X_\infty}<C_0.
\end{align}
\end{propal}
\begin{proof}
By Theorem \ref{unTheorem:localexist}, there exists a local solution $u\in C([1,T],H^{1,0}_x)$ for some $T>0$ with the corresponding bound 
\begin{align}
||u||_{X_T}\lesssim C_X(\varepsilon_0).
\end{align}
Then, with proposition \ref{propal:bootstrap}, we apply the bootstrap argument to show that, for $\varepsilon_0>0$ sufficiently small, the local existence interval can be extended indefinitely. 
\end{proof}
\section{Asymptotic analysis}
\label{section:asymptotic}
In this section, we derive the expansion of the profile $\hat{f}$, the modified profile $\hat{w}$ and the solution $u$ itself. Because we aim at an expansion of order $N$, we decompose the full nonlinearity into a truncated sum and a remainder term:
\begin{equation}
    i\partial_t\hat{f}(t,\xi)=\sum_{n=1}^{N+1}\lambda_n\hat{\mathfrak{P}}^n(t,\xi)+\mathfrak{R}^{> N+1}_0(t,\xi).
\end{equation}
Then, we expand each contribution $\mathfrak{P}^n$ using the stationary phase, leading to 
\begin{equation}
    i\partial_t\hat{f}(t,\xi)=\sum_{n=1}^{N+1}\lambda_n\sum_{k=0}^{N+1-n}\frac{1}{t^{k+n}}\hat{\mathfrak{P}}^n_k(t,\xi)+\sum_{n=1}^{N+1}\mathfrak{R}^n_{N+2-n}(t,\xi)+\mathfrak{R}^{> N+1}_0(t,\xi),
\end{equation}
where the quantities $\mathfrak{P}^n_k$ are defined in lemma \ref{lem:exactformP}. Reorganizing the expansion according to powers of $\frac{1}{t}$ yields
\begin{equation}
\label{eq:organizedexp}
    i\partial_t\hat{f}(t,\xi)=\sum_{m=1}^{N+1}\frac{1}{t^m}(\sum_{k+n=m,n\geq1}\lambda_n\hat{\mathfrak{P}}^n_k(t,\xi))+\sum_{n=1}^{N+1}\mathfrak{R}^n_{N+2-n}(t,\xi)+\mathfrak{R}^{> N+1}_0(t,\xi).
\end{equation}
As emphasized in the introduction, the case $N=0$ is of particular importance. Indeed, for $N=0$, one obtains
\begin{equation}
    i\partial_t\hat{f}(t,\xi)=-\lambda_1\frac{1}{t}|\hat{f}(t,\xi)|^2\hat{f}(t,\xi)+\mathfrak{R}^1_{1}(t,\xi)+\mathfrak{R}^{>1}_0(t,\xi).
\end{equation}
This motivates the introduction of the modified profile $\hat{w}(t,\xi)=\Theta(t,\xi)\hat{f}(t,\xi)$, with $\Theta(t,\xi)=e^{i\int^t_1\frac{\lambda_1|\hat{f}(s,\xi)|^2}{s}ds}$, which satisfies 
\begin{equation}
    i\partial_t\hat{w}(t,\xi)=\Theta(t,\xi)\mathfrak{R}^1_{1}(t,\xi)+\Theta(t,\xi)\mathfrak{R}^{>1}_0(t,\xi),
\end{equation}
or, for arbitrary order $N$, 
\begin{equation}
\label{eq:organizedexpw}
    i\partial_t\hat{w}(t,\xi)=\sum_{m=2}^{N+1}\frac{1}{t^m}(\sum_{k+n=m,n\geq1}\lambda_n\Theta(t,\xi)\hat{\mathfrak{P}}^n_k(t,\xi))+\sum_{n=1}^{N+1}\Theta(t,\xi)\mathfrak{R}^n_{N+2-n}(t,\xi)+\Theta(t,\xi)\mathfrak{R}^{> N+1}_0(t,\xi).
\end{equation}
Our next goal is to analyze the asymptotic expansion of the modified profile $\hat{w}$ as $t\to+\infty$, and then to translate it back to $\hat{f}$. This procedure provides us the phase correction of the modified scattering. Finally, we transfer the information on the expansion of $\hat{f}$ into the expansion for the solution $u$. \\\\
In each term $\Theta(t,\xi)\hat{\mathfrak{P}}^n_k(t,\xi)$, each occurrence of $\hat{f}(t,\xi)$, $\overline{\hat{f}(t,\xi)}$, or their derivatives in $\xi$, can be replaced by $\hat{w}(t,\xi)$, $\overline{\hat{w}(t,\xi)}$, or their derivatives in $\xi$, at the expense of commutators arising when derivatives fall on $\Theta(t,\xi)$. Indeed, there are $n+1$ occurrences of $\hat{f}(t,\xi)$ and $n$ occurrences of $\overline{\hat{f}(t,\xi)}$, and we know that $\Theta^{-1}(t,\xi)=\overline{\Theta}(t,\xi)$. Thus, equation \eqref{eq:organizedexpw} may be written entirely in terms of $\hat{w}$. This is done in the following lemma. 
\begin{lem}
\label{lem:nonlinearityPexactform}
We have 
\begin{align}
    \Theta(t,\xi)\hat{\mathfrak{P}}^n_k(t,\xi)=\frac{i^k}{k!}\sum_{\overrightarrow{m}\in\mathbb{N}^n,|\overrightarrow{m}|=k}\binom{k}{\overrightarrow{m}}
A^n_{\overrightarrow{m}}(t,\xi).
\end{align}
where  
\begin{align}
\label{eq:Adefined}
    A^n_{\overrightarrow{m}}(t,\xi)=\sum_{(J,K)\in\mathscr{L}_{\overrightarrow{m},\overrightarrow{m}}}\prod^n_{j=1}\binom{\overrightarrow{m}_j}{K_j}\binom{\overrightarrow{m}_j}{J_j}D^{K_{j,1}+\sum_{l=1}^{n}J_{l,2j+1}}\hat{w}(\xi)\overline{D^{K_{j,2}+\sum_{l=1}^{n}J_{l,2j}}\hat{w}(\xi)}D^{\sum_{l=1}^{n}J_{l,1}}\hat{w}(\xi)
\end{align}
for 
\begin{align}
\label{eq:operatorDk}
    D^k=(-1)^k\sum_{l=0}^k\binom{k}{l}(\sum_{j=0}^lB_{l,j}(\overline{\partial_\xi F},\overline{\partial^2_\xi F},\dots,\overline{\partial_\xi^{(l-j+1)}F}))\partial_\xi^{k-l}
\end{align}
with $F(t,\xi)=i\int^t_1\frac{\lambda_1|\hat{f}(s,\xi)|^2}{s}ds$ and where $B_{l,j}$ represents Bell's polynomials.\\
\end{lem}
\begin{rem}
\label{rem:nonlinearityPexactform}
In particular, each occurrence of $\hat{f}$ in $F$ and its derivatives in $\xi$ (entering in the definition of $
A_{\overrightarrow{m}}$ in lemma \ref{lem:nonlinearityPexactform}) can be replaced by an occurrence of $\hat{w}$, since $|\hat{f}|^2=|\hat{w}|^2$.
\end{rem}
\begin{proof}
Recall that 
\begin{align}
    \hat{\mathfrak{P}}^n_k(t,\xi)=\frac{i^k}{k!}\sum_{\overrightarrow{m}\in\mathbb{N}^n,|\overrightarrow{m}|=k}\binom{k}{\overrightarrow{m}}
\partial^{\overrightarrow{m},\overrightarrow{m}}\mathcal{N}^n(t,\xi,0),
\end{align}
where the explicit form of $\partial^{\overrightarrow{m},\overrightarrow{m}}\mathcal{N}^n(t,\xi,0)$ is given in Lemma \ref{lem:exactformN}. By definition, $\hat{f}(t,\xi)=\Theta^{-1}(t,\xi)\hat{w}(t,\xi)=e^{\overline{F(t,\xi)}}\hat{w}(t,\xi)$. Hence, for all $k\in\mathbb{N}$, we deduce that
\begin{align*}
\partial^k_\xi\hat{f}(t,\xi)=\sum_{l=0}^k\Theta^{-1}(t,\xi)\binom{k}{l}(\sum_{j=0}^lB_{l,j}(\overline{F'},\overline{F^{\prime\prime}},\dots,\overline{F^{(l-j+1)}}))\partial_\xi^{k-l}\hat{w}(t,\xi),
\end{align*}
where the $B$'s are Bell's polynomials arising from Faà di Bruno's formula. We also have 
\begin{align*}
\partial^k_\xi\overline{\hat{f}(t,\xi)}=\sum_{l=0}^k\Theta(t,\xi)\binom{k}{l}(\sum_{j=1}^lB_{l,j}(F',F^{\prime\prime},\dots,F^{(l-j+1)}))\partial_\xi^{k-l}\overline{\hat{w}(t,\xi)},
\end{align*}
so that if using the operator $D^k$ defined above, 
\begin{align}
&(-\partial)^k_\xi\hat{f}(t,\xi)=\Theta^{-1}(t,\xi)D^k\hat{w}(t,\xi), &(-\partial)^k_\xi\overline{\hat{f}(t,\xi)}=\overline{\Theta^{-1}(t,\xi)D^k\hat{w}(t,\xi)}.
\end{align}
This implies that 
\begin{align*}
    &\Theta(t,\xi)\partial^{\overrightarrow{m},\overrightarrow{m}}\mathcal{N}^n[\hat{w}](\xi,0)\\
    &=\sum_{(J,K)\in\mathscr{L}_{\overrightarrow{\mu},\overrightarrow{\nu}}}\prod^n_{j=1}\binom{\overrightarrow{m}_j}{K_j}\binom{\overrightarrow{m}_j}{J_j}D^{K_{j,1}+\sum_{l=1}^{n}J_{l,2j+1}}\hat{w}(\xi)\overline{(-D)^{K_{j,2}+\sum_{l=1}^{n}J_{l,2j}}\hat{w}(\xi)}D^{\sum_{l=1}^{n}J_{l,1}}\hat{w}(\xi)
\end{align*}
using the formula of lemma \ref{lem:exactformN}. This leads to the desired equality.
\end{proof}
\subsection{Asymptotics for the modified profile}
\label{subsection:asymptoticmodif}
In this section, we derive the asymptotics of the modified profile $\hat{w}$ as $t\to+\infty$.\\
\begin{propal}
\label{propal:bigpropalexpw}
Under the assumption of Theorem \ref{unTheorem:mainTH}, let $u$ be the global solution to \eqref{eq:NLS} given by proposition \ref{propal:GWP}. Then, the modified profile $\hat{w}$ admits the following asymptotic expansion: for every $0\leq l\leq N$
\begin{equation}
    \hat{w}(t,\xi)= \sum^{l}_{j=0}\sum^{2j}_{k=0}\frac{\ln(t)^k}{t^j}\hat{w}_{j,k}(\xi)+r^{l}(t,\xi)
\end{equation}
where $||r^l(t)||_{W^{2(N-l),\infty}_\xi}=O(t^{-l-\beta})$ and where the coefficients $(\hat{w}_{j,k})_{(j,k)\in\mathbb{N}^{2},\;j\leq N,\;k\leq 2j}$ satisfy $\hat{w}_{j,k}\in W^{2(N-j),\infty}_\xi$. 
\end{propal}
The proof proceeds by induction. First, we establish the existence of the zeroth-order expansion with sufficient control on the derivatives of the expansion coefficients. This is done in lemma \ref{lem:mainofbasecase}. Then, assuming that the expansion exists up to order $(n-1)$ with suitable bounds on the derivatives of the coefficients and the error terms, we construct the $n$-th term of the asymptotic expansion of $\hat{w}$ (for $1\leq n\leq N$) with the corresponding proper control. This is achieved in proposition \ref{propal:mainofassympt}. \\\\
Throughout this section, $u$ denotes the global solution to \eqref{eq:NLS} given by proposition \ref{propal:GWP} under the assumption of Theorem \ref{unTheorem:mainTH}, $f$ is its profile and $\hat{w}$ the associated modified profile. \\
\begin{lem}
\label{lem:mainofbasecase}
There exists $\hat{w}_{0,0}\in W^{2N,\infty}_\xi$ such that 
\begin{align}
\hat{w}(t,\xi)=\hat{w}_{0,0}(\xi)+r^0(t,\xi),
\end{align}
with $r^0(t)\in W^{2N,\infty}_\xi$.
\end{lem}
\begin{proof}
We use \eqref{eq:organizedexpw} at first order $N=0$, that is, 
\begin{equation}
\label{eq:basecase}
    i\partial_t\hat{w}(t,\xi)=\Theta(t,\xi)\mathfrak{R}^1_{1}(t,\xi)+\Theta(t,\xi)\mathfrak{R}^{>1}_0(t,\xi).
\end{equation}
Commuting \eqref{eq:basecase} with $\partial^q_\xi$ for $0\leq q\leq 2N$ and integrating between $s$ and $t\geq s$, proposition \ref{propal:controlerror} yields
\begin{align*}
\partial^q_\xi\hat{w}(t,\xi)-\partial^q_\xi\hat{w}(s,\xi)&=-i\int^t_s\partial^q_\xi(\Theta(\tau,\xi)\mathfrak{R}^1_{1}(\tau,\xi))+\partial^q_\xi(\Theta(\tau,\xi)\mathfrak{R}^{>1}_0(\tau,\xi))d\tau\\
&=\int^t_s O(\tau^{-1-\delta})d\tau\\
&=O(t^{-\delta}),
\end{align*}
so that $(\hat{w}(t))_{t\geq0}$ is a Cauchy sequence in $W^{2N,\infty}_\xi$ and converges to some $\hat{w}_{0,0}\in W^{2N,\infty}_\xi$.
\end{proof}
We now assume that the asymptotic expansion is known up to order $(n-1)$ with proper control on the error and its derivatives.\\ 
\begin{ass}
\label{ass:lowerorderexpansion}
For all $0\leq l\leq(n-1)$, 
\begin{equation}
\label{eq:expansionn-un}
    \hat{w}(t,\xi)= \sum^{l}_{j=0}\sum^{2j}_{k=0}\frac{\ln(t)^k}{t^j}\hat{w}_{j,k}(\xi)+r^{l}(t,\xi)
\end{equation}
where $||r^l(t)||_{W^{2(N-l),\infty}_\xi}=O(t^{-l-\beta})$ and where the coefficients $(\hat{w}_{j,k})_{(j,k)\in\mathbb{N}^{2},\;j\leq n-1,\;k\leq 2j}$ satisfy $\hat{w}_{j,k}\in W^{2(N-j),\infty}_\xi$. \\
\end{ass}
 \begin{rem}
 \label{rem:limitexpansionderivw}
Higher-order terms of the expansion of $\hat{w}$ require control over fewer derivatives. The reason for that is that, for $\hat{w}$ to appear with high-order derivatives in \eqref{eq:organizedexpw}, it must appear in a term of high-order in $t^{-1}$. Concretely, $\partial^{2l}\hat{w}$ only appears in $\mathfrak{P}^k_{l'}$ for $k\geq1$, $l'\geq l$, which itself appears at order $t^{-l-1}$ at least. Therefore, to construct the expansion at order $n$, it suffices to know the expansion of $\partial^{2l}\hat{w}$ up to order $n-l$, with extra control on $2(N-n+l)$ derivatives to later reach the final order $N$. This precisely matches the derivative hierarchy of assumption \ref{ass:lowerorderexpansion}.\\
 \end{rem}

\begin{propal}
\label{propal:mainofassympt}
Under the assumption \ref{ass:lowerorderexpansion}, for every $0\leq l\leq n$ one has
\begin{equation}
    \hat{w}(t,\xi)= \sum^{l}_{j=0}\sum^{2j}_{k=0}\frac{\ln(t)^k}{t^j}\hat{w}_{j,k}(\xi)+r^{l}(t,\xi)
\end{equation}
where $||r^l(t)||_{W^{2(N-l),\infty}_\xi}=O(t^{-l-\beta})$ and where $(\hat{w}_{j,k})_{(j,k)\in\mathbb{N}^{2},\;j\leq n,\;k\leq 2j}$ satisfy $\hat{w}_{j,k}\in W^{2(N-j),\infty}_\xi$. 
\end{propal}
The proof of proposition \ref{propal:mainofassympt} is given after precision and a series of Lemmas. The basic idea of the proof is to replace all occurrences of $ \hat{w}$ in the RHS of \eqref{eq:organizedexpw} by the $(n-1)$ or lower-order expansion \eqref{eq:expansionn-un}, and then integrate in time to gain one additional order and a smaller remainder.   \\\\
The RHS of \eqref{eq:organizedexpw} is composed of terms contributing to the expansion up to order $n\leq N$, and error terms. We rewrite \eqref{eq:organizedexpw}, replacing $N$ with $n$, in the less precise but more compact form 
\begin{equation}
\label{eq:organizedexpwcompact}
      i\partial_t\hat{w}(t,\xi)=\sum_{m=2}^{n+1}\frac{1}{t^m}\mathfrak{F}^m[\hat{w}](t,\xi)+\mathfrak{R}^{n+1}_{\text{tot.}}(t,\xi),
\end{equation}
where 
\begin{align}
\label{eq:Ffrakequation}
    \mathfrak{F}^m[\hat{w}](t,\xi)=(\sum_{k+l=m,l\geq1}\lambda_l\frac{i^k}{k!}\sum_{\overrightarrow{m}\in\mathbb{N}^l,|\overrightarrow{m}|=k}\binom{k}{\overrightarrow{m}}
A^l_{\overrightarrow{m}}(t,\xi))
\end{align}
and \begin{align}
    \mathfrak{R}^{n+1}_{\text{tot.}}(t,\xi)=\sum_{l=1}^{n+1}\Theta(t,\xi)\mathfrak{R}^l_{n+2-l}(t,\xi)+\Theta(t,\xi)\mathfrak{R}^{>n+1}_0.
\end{align}
 As suggested in remark \ref{rem:nonlinearityPexactform}, we express $\mathfrak{F}^m$ entirely in terms of $\hat{w}$ and not $\hat{f}$. In general, for the exact form of $\mathfrak{F}^m[\hat{w}]$, we refer to lemma \ref{lem:nonlinearityPexactform}. We now describe the expansion of $F(t,\xi)$, the phase shift that appears in the operator $D^k$ defined in \eqref{eq:operatorDk} in Lemma \ref{lem:nonlinearityPexactform}, with respect to the $(n-1)$ or lower order-expansion \eqref{eq:expansionn-un}.\\ 
 \begin{lem}
 \label{lem:expansionF}
Under the assumption \ref{ass:lowerorderexpansion}, for all $0\leq l\leq(n-1)$, we have 
 \begin{align*}
     F(t,\xi)=F_{0,0}(\xi)+\ln(t)F_{0,1}(\xi)+\sum^{l}_{j=1}\sum^{2j}_{k=0}\frac{\ln(t)^{k}}{\tau^{j}}F_{j,k}(\xi)+ r^{l}_F(t,\xi),
 \end{align*}
where $||r^l_F(t)||_{W^{2(N-l),\infty}_\xi}=O(t^{-l-\beta})$ and where $(F_{j,k})_{(j,k)\in\mathbb{N}^{2},\;j\leq n-1,\;k\leq 2j}$ satisfy $F_{j,k}\in W^{2(N-j),\infty}_\xi$. In particular, 
\begin{align}
&F_{0,1}(\xi)=i\lambda_1|\hat{w}_{0,0}(\xi)|^2, &F_{0,0}(\xi)=i\int^\infty_1\frac{\lambda_1|\hat{w}(\tau,\xi)|^2-|\hat{w}_{0,0}(\xi)|^2}{\tau}d\tau.
\end{align}
 \end{lem}
 \begin{rem}
The expansion of $F$ holds under the same conditions as those of the expansion of $\hat{w}$ in assumption \ref{ass:lowerorderexpansion}. In particular, remark \ref{rem:limitexpansionderivw} holds for $\hat{w}$ replaced by $F$. 
 \end{rem}
\begin{proof}[Proof of Lemma \ref{lem:expansionF}]
Let $0\leq l\leq(n-1)$. The proof follows by inserting the expansion of $\hat{w}$ up to order $l$ into the definition
\begin{align*}
  F(t,\xi)&=i\int^t_1\frac{\lambda_1|\hat{w}(\tau,\xi)|^2}{\tau}d\tau.
  \end{align*}
  We start by isolating the logarithmic divergence term
\begin{align*}
  F(t,\xi)&=i\int^t_1\frac{\lambda_1(|\hat{w}(\tau,\xi)|^2-|\hat{w}_{0,0}(\xi)|^2)}{\tau}d\tau+\ln(t)F_{0,1}(\xi),
  \end{align*}
where
\begin{align}
F_{0,1}(\xi)=i\lambda_1|\hat{w}_{0,0}(\xi)|^2.
\end{align}
Now, using the expansion for $\hat{w}$ up to order $0$ given in assumption \ref{ass:lowerorderexpansion}, we clearly see that $i\int^t_1\frac{\lambda_1|\hat{w}(\tau,\xi)|^2-|\hat{w}_{0,0}(\xi)|^2}{\tau}d\tau$ converges to some function $F_{0,0}\in W^{2N,\infty}$ as $t\to\infty$. More precisely,
\begin{align}
F_{0,0}(\xi)=i\int^\infty_1\frac{\lambda_1(|\hat{w}(\tau,\xi)|^2-|\hat{w}_{0,0}(\xi)|^2)}{\tau}d\tau.
\end{align}
The $W^{2N,\infty}_\xi$ regularity for $F_{0,1}$ and $F_{0,0}$ follows from the regularity of the zeroth-order expansion of $\hat{w}$, which is the only order involve in their expression.\\\\
To compute the next orders, we write 
\begin{align}
\label{eq:noidea}
  F(t,\xi)&=\ln(t)F_{0,1}(\xi)+F_{0,0}(\xi)-i\int^\infty_t\frac{\lambda_1(|\hat{w}(\tau,\xi)|^2-|\hat{w}_{0,0}(\xi)|^2)}{\tau}d\tau.
  \end{align}
Then, using assumption \ref{ass:lowerorderexpansion} to get the expansion of $\hat{w}$ at order $l$ and plugging it into \eqref{eq:noidea} yields
\begin{align*}
  F(t,\xi)-\ln(t)F_{0,1}(\xi)-F_{0,0}(\xi)=&-i\int^\infty_t\frac{\lambda_1}{\tau}\left(\sum^{l}_{j_1=0}\sum^{2j_1}_{k_1=0}\sum^{l}_{j_2=0}\sum^{2j_2}_{k_2=0}\frac{\ln(\tau)^{k_1+k_2}}{\tau^{j_1+j_2}}\hat{w}_{j_1,k_1}(\xi)\overline{\hat{w}_{j_2,k_2}(\xi)}-|\hat{w}_{0,0}(\xi)|^2\right)d\tau\\
    &-i\int^\infty_t\frac{\lambda_1}{\tau}\left(\sum^{l}_{j=0}\sum^{2j}_{k=0}\frac{\ln(\tau)^{k}}{\tau^{j}}(r^{l}(\tau,\xi)\overline{\hat{w}_{j,k}(\xi)}+\overline{r^{l}(\tau,\xi)}\hat{w}_{j,k}(\xi))\right)d\tau\\
        &-i\int^{\infty}_t\frac{\lambda_1}{\tau}r^{l}(\tau,\xi)\overline{r^{l}(\tau,\xi)}d\tau.
\end{align*}
Thus, we deduce that 
\begin{align*}
 F(t,\xi)-\ln(t)F_{0,1}(\xi)-F_{0,0}(\xi)
    =&\left[-i\lambda_1\left(\sum^{l}_{j_1=0}\sum^{2j_1}_{k_1=0}\sum^{l}_{j_2=0}\sum^{2j_2}_{k_2=0}L^{k_1+k_2}_{j_1+j_2+1}(\tau)\hat{w}_{j_1,k_1}(\xi)\overline{\hat{w}_{j_2,k_3}(\xi)}-\ln(\tau)|\hat{w}_{0,0}(\xi)|^2\right)\right]^{\infty}_t\\
    &+\left[ r'^l_F(\tau,\xi)\right]^{\infty}_t
\end{align*}
where $L^k_{j}(\tau)=\int^\tau\frac{\ln(\sigma)^k}{\sigma^j}d\sigma$ and where $||r'^l_F(\tau)||_{W^{2(N-l),\infty}_\xi}=O(\tau^{-p-\beta})$. Then, for $j\geq1$, explicit computations lead to
   \begin{equation}
   \nonumber
    \begin{cases}
            L^k_{j}(\tau)=\frac{\ln(\tau)^{k+1}}{k+1},\;\; \text{if $j=1$},\\
            L^k_{j}(\tau)=\sum^{k}_{k'=0}(-1)^{k'}\binom{k}{k'}(k')!\frac{\ln(\tau)^{k-k'}}{(1-j)^{k'+1}\tau^{j-1}},\;\; \text{if $j>1$},
    \end{cases}
    \end{equation}
 from which we deduce the form 
\begin{align}
\label{eq:deroulF}
   F(t,\xi)-\ln(t)F_{0,1}(\xi)-F_{0,0}(\xi)
    &=\sum^{l}_{j=1}\sum^{2j}_{k=0}\frac{\ln(t)^{k}}{t^{j}}F_{j,k}(\xi)+r^l_F(t)
\end{align}
where $r^l_F$ satisfy $||r^l_F(t)||_{W^{2(N-l),\infty}_\xi}=O(t^{-l-\beta})$. This is the desired result.
\end{proof}
Now knowing the expansion of $\hat{w}$ up to order $n-1$ by assumption \ref{ass:lowerorderexpansion} and the expansion of $F$ up to the same order as its consequence with lemma \ref{lem:expansionF}, we may expand the RHS of \eqref{eq:organizedexpwcompact} to get the expansion of $\hat{w}$ up to order $n$. \\
\begin{lem}
\label{lem:expansionA}
Suppose that assumption \ref{ass:lowerorderexpansion} holds. Let $2\leq q\leq n+1$. Then, for $1\leq l$, $\overrightarrow{m}\in\mathbb{N}^l$, with $|\overrightarrow{m}|+l=q$ and $p\leq n+1-q$, 
\begin{align}
    A^l_{\overrightarrow{m}}(t,\xi)=\sum_{j=0}^p\sum_{k=0}^{2|\overrightarrow{m}|}\frac{\ln(t)^k}{t^j}(A^l_{\overrightarrow{m}})_{j,k}(\xi)+r^p_{A^l_{\overrightarrow{m}}}(t,\xi)
\end{align}
for a family $((A^l_{\overrightarrow{m}})_{j,k})_{(j,k)\in\mathbb{N}^{2},\;j\leq p,\;k\leq 2j}$ satisfying $(A^l_{\overrightarrow{m}})_{j,k}\in W^{2(N-j-|\overrightarrow{m}|),\infty}_\xi$ and an error $r^p_{A^l_{\overrightarrow{m}}}$ satisfying  \\$||r^p_{A^l_{\overrightarrow{m}}}(t,\xi)||_{W^{2(N-p-|\overrightarrow{m}|),\infty}_\xi}=O(t^{-p-\beta})$.
\end{lem}
\begin{rem}
In Lemma \ref{lem:expansionA}, the index $l$ corresponds to the nonlinearity with power $2l+1$ in the original equation \eqref{eq:NLS}. Consequently, the quantity $A^l_{\overrightarrow{m}}$ contains a product of $2l+1$ factors of $\hat{w}$. The multi-index $\overrightarrow{m}$ is such that $2|\overrightarrow{m}|$ represents the total number of derivatives falling on these $\hat{w}$ factors. We refer to Lemma \ref{lem:nonlinearityPexactform} for the precise definition of $A^l_{\overrightarrow{m}}$. We see that $A^l_{\overrightarrow{m}}$ appears at order $t^{-|\overrightarrow{m}|-l}$ in the equation \eqref{eq:organizedexpwcompact} and therefore must admit an expansion up to order $n+1-|\overrightarrow{m}|-l$ in $W^{2(N-n),\infty}_\xi$. This is reflected in the regularity statement of Lemma \ref{lem:expansionA} for the coefficients $((A^l_{\overrightarrow{m}})_{j,k})_{(j,k)\in\mathbb{N}^{2},\;j\leq p,\;k\leq 2j}$ and the error $r^p_{A^l_{\overrightarrow{m}}}$. The required regularity follows from the assumptions on $\hat{w}$ in assumption \ref{ass:lowerorderexpansion} and the properties of $F$ in lemma \ref{lem:expansionF}. Indeed, $2(N-p-|\overrightarrow{m}|)\geq 2(N-n-1+q-|\overrightarrow{m}|)\geq2(N-n)$, since $|\overrightarrow{m}|=q-l\leq q-1$. We therefore see that $|\overrightarrow{m}|$ is the limiting parameter, as it represents the number of derivatives on $\hat{w}$.
\end{rem}
\begin{proof}[Proof of Lemma \ref{lem:expansionA}]
The proof follows by direct expansion. We use the explicit expression  \eqref{eq:Adefined} that defines $A^l_{\overrightarrow{m}}$, together with the expansion of $\hat{w}$, given in assumption \ref{ass:lowerorderexpansion}, and the expansion for $F$ provided by Lemma \ref{lem:expansionF}. In particular, the operator $D^{2|\overrightarrow{m}|}$ shows that at most  $2|\overrightarrow{m}|$ occurrences of $F$ appear in the product defining $A^l_{\overrightarrow{m}}$. Since each occurrence of $F$ yields at most one logarithmic growth in time, the total number of logarithm of $t$ is at most $2|\overrightarrow{m}|$. 
\end{proof}
\begin{proof}[Proof of proposition \ref{propal:mainofassympt}]
In \eqref{eq:organizedexpwcompact}, the terms $\mathfrak{F}^m[\hat{w}](s,\xi)$, defined by \eqref{eq:Ffrakequation}, can be expanded using the expansion of $A^l_{\overrightarrow{m}}(t,\xi)$ given by Lemma \ref{lem:expansionA}. Integrating \eqref{eq:organizedexpwcompact} between $t$ and $+\infty$, we deduce that 
\begin{align*}
    \hat{w}(t,\xi)&=\hat{w}(+\infty,\xi)+i\int^{+\infty}_t\sum_{m=2}^{n+1}\frac{1}{t^m}\mathfrak{F}^m[\hat{w}](s,\xi)ds+i\int^{+\infty}_t\mathfrak{R}^{n+1}_{\text{tot.}}(s,\xi)ds\\
&=\hat{w}_{0,0}(\xi)+\sum_{j=1}^{n}\sum^{2j}_{k=0}\frac{\ln(t)^k}{t^j}\hat{w}^{\prime}_{j,k}(\xi)+r'^{n}(t,\xi)+r^{\prime\prime n}(t,\xi)\\
&=\hat{w}_{0,0}(\xi)+\sum_{j=1}^{n}\sum^{2j}_{k=0}\frac{\ln(t)^k}{t^j}\hat{w}_{j,k}(\xi)+r^{n}(t,\xi).
\end{align*}
The existence of the expansion coefficients 
$(\hat{w}_{j,k}^\prime)_{(j,k)\in\mathbb{N}^{2},\;j\leq n,\;k\leq 2j}$ satisfying $\hat{w}_{j,k}\in  W^{2(N-j),\infty}_\xi$ follows directly from Lemma \ref{lem:expansionA} under assumption \ref{ass:lowerorderexpansion}. Then, by uniqueness of the asymptotic expansion, the coefficients previously constructed must coincide:
$(\hat{w}_{j,k}^\prime)_{(j,k)\in\mathbb{N}^{2},\;j\leq n-1,\;k\leq 2j}$ must be equal to $(\hat{w}_{j,k})_{(j,k)\in\mathbb{N}^{2},\;j\leq n-1,\;k\leq 2j}$ term by term. Finally, we may define $(\hat{w}_{n,k})_{k\in\mathbb{N},\; k\leq2n}=(\hat{w}_{n,k}^\prime)_{k\in\mathbb{N},\; k\leq2n}$ which provides the expansion coefficients at order $n$. \\\\
From lemma \ref{lem:expansionA} and proposition \ref{propal:controlerror}, there exists $0<\delta<1$ such that $||r'^{n}||_{W^{2(N-n),\infty}_\xi}=O(s^{-n-\delta})$ and $||r^{\prime\prime n}||_{W^{2(N-n),\infty}_\xi}=O(s^{-n-\delta})$, respectively. More precisely, to estimate $\mathfrak{R}^{n+1}_{\text{tot.}}$, we have to control a finite sum of $\Theta\mathfrak{R}^l_k$ with $l+k=n+2$ and $l\geq1$ and $\Theta\mathfrak{R}^l_k$ with $k=0$ and $l>n+1$. Proposition \ref{propal:controlerror} shows that, under the bounds $||\hat{w}||_{L^\infty W^{2N,\infty}_\xi}$ and $||t^{-\alpha_{2N+1}}f||_{L^\infty H^{0,2N+1}_x}$ one can control up to $2(N-n)$ derivatives\footnote{The limiting case is $k=n+1$, for which $2k-2+q=2N$, yielding $q=2N-2n$.} of $\mathfrak{R}^{n+1}_{\text{tot.}}$, which is exactly the regularity required here. Since $l+k\geq n+2$, we obtain $||\mathfrak{R}^{n+1}_{\text{tot.}}(t)||_{W^{2(N-n),\infty}_\xi}=O(t^{-n-1-\delta})$. Then, we lose $t^{-1}$ after the integration in time, which precisely the decay needed for $r^{\prime\prime n}$. Finally, we set $r^{n}:=r'^{n}+r^{\prime\prime n}$, which complete the proof.
\end{proof} 
\subsection{Asymptotics for the profile}
\label{subsection:asymptoticprof}
In this section, we derive the asymptotics of the profile $\hat{f}$ as $t\to+\infty$. \\
\begin{propal}
\label{propal:bigpropalexpf}
Under the assumption of Theorem \ref{unTheorem:mainTH}, let $u$ be the global solution to \eqref{eq:NLS} given by proposition \ref{propal:GWP}. Then, the profile $\hat{f}$ admits the  following expansion: for every $0\leq l\leq N$
\begin{equation}
    \hat{f}(t,\xi)= e^{-i\ln(t)\lambda_1|\hat{w}_{0,0}(\xi)|^2-F_{0,0}(\xi)}\left(\sum_{j=0}^l\sum_{k=0}^{2j}\frac{\ln(t)^{k}}{t^{j}}\hat{f}(\xi)_{j,k}\right)+r^l_{f}(t,\xi)
\end{equation}
where $||r^l_f(t)||_{W^{2(N-l),\infty}_\xi}=O(t^{-l-\beta})$ and where $(\hat{f}_{j,k})_{(j,k)\in\mathbb{N}^{2},\;j\leq N,\;k\leq 2j}$ satisfy $\hat{f}_{j,k}\in W^{2(N-j),\infty}_\xi$. In particular, $\hat{f}_{0,0}=\hat{w}_{0,0}$.
\end{propal}
The expansion of order $N$ of the profile $\hat{f}=\overline{\Theta(t,\xi)}\hat{w}(t,\xi)$, with $\Theta(t,\xi)=e^{F(t,\xi)}$, is directly deduced from the expansion of the modified profile $\hat{w}(t,\xi)$ up to order $N$, with proposition \ref{propal:bigpropalexpw}, and the expansion of the phase shift $F$ up to order $N$, with lemma\footnote{We apply the lemma \ref{lem:expansionF}, not under the assumption \ref{ass:lowerorderexpansion}, but under the fact that proposition \ref{propal:bigpropalexpw} holds at order $N$.} \ref{lem:expansionF}.\\
\begin{lem}
\label{lem:expansionTheta}
For every $0\leq l\leq N$, we have
\begin{equation}
\Theta(t,\xi)=e^{i\ln(t)\lambda_1|\hat{w}_{0,0}(\xi)|^2+F_{0,0}(\xi)}(\sum_{j=0}^l\sum_{k=0}^{2j}\frac{\ln(t)^{k}}{t^{j}}\Theta_{j,k}(\xi))+r^l_{\Theta}(t,\xi)
\end{equation}
where $||r^l_\Theta(t)||_{W^{2(N-l),\infty}_\xi}=O(t^{-l-\beta})$ and where $(\Theta_{j,k})_{(j,k)\in\mathbb{N}^{2},\;j\leq N,\;k\leq 2j}$ satisfy $\Theta_{j,k}\in W^{2(N-j),\infty}_\xi$. 
\end{lem}
\begin{proof}
For $0\leq l\leq N$, we compute first 
\begin{align*}
\Theta(t,\xi)=&e^{F(t,\xi)}\\
=&\underbrace{e^{i\ln(t)\lambda_1|\hat{w}_{0,0}(\xi)|^2+F_{0,0}(\xi)}(\sum_{j=0}^l\frac{1}{j!}\left(F(t,\xi)-i\ln(t)\lambda_1|\hat{w}_{0,0}(\xi)|^2-F_{0,0}(\xi)\right)^j)}_\text{$\Theta'(t,\xi)$}\\
&+\underbrace{e^{i\ln(t)\lambda_1|\hat{w}_{0,0}(\xi)|^2+F_{0,0}(\xi)}(e^{F(t,\xi)-i\ln(t)\lambda_1|\hat{w}_{0,0}(\xi)|^2-F_{0,0}(\xi)}-\sum_{j=0}^l\frac{1}{j!}\left(F(t,\xi)-i\ln(t)\lambda_1|\hat{w}_{0,0}(\xi)|^2-F_{0,0}(\xi)\right)^j)}_\text{$r'^l_{\Theta}(t,\xi)$}.
\end{align*}
Using the $0$-th order expansion of $F$ from Lemma \ref{lem:expansionF}, the error satisfies $||r'^l_{f}(t)||_{W^{2N,\infty}_\xi}=O(t^{-l-\delta})$ for some $0<\delta<1$. Then, for the $j$-th term of the Taylor expansion in $\Theta'$, we use the expansion of $F$ up to order $l-j+1$ to have
\begin{align*}
\Theta'(t,\xi)&=e^{i\ln(t)\lambda_1|\hat{w}_{0,0}(\xi)|^2+F_{0,0}(\xi)}(\sum_{j=0}^l\frac{1}{j!}\left(\sum^{l-j+1}_{k=1}\sum^{2k}_{m=0}\frac{\ln(t)^{m}}{t^{k}}F_{k,m}(\xi)+ r^{l-j+1}_F(t,\xi)\right)^j).
\end{align*}
Then, for $j\geq1$, we compute that
\begin{align*}
\left(\sum^{l-j+1}_{k=1}\sum^{2k}_{m=0}\frac{\ln(t)^{m}}{t^{k}}F_{k,m}(\xi)+ r^{l-j+1}_F(t,\xi)\right)^j=&\sum_{\overrightarrow{\mu}\in\mathbb{N}^{(l-j+1)},|\overrightarrow{\mu}|=j,\;J\in\mathscr{J}_{\overrightarrow{\mu}}}\binom{j}{J}\left(\prod^{l-j+1}_{k=1}\prod^{2k}_{m=0}(\frac{\ln(t)^{m}}{t^{k}}F_{k,m}(\xi))^{J_{k,m}}\right)\\
&+r^{l,j}_\Theta(t,\xi)
\end{align*}
where $r^{l,j}_F(t)\in W^{2(N-l+j-1),\infty}_\xi$ and where $\mathscr{J}_{\overrightarrow{\mu}}$ is from definition \ref{defi:config}, with $\binom{j}{J}=\frac{j!}{\prod^{l-j+1}_{k=1}\prod^{2k}_{m=0}J_{k,m}!}$. Then, rearranging the different terms by powers of $t^{-1}$ gives 
\begin{align*}
\left(\sum^{l-j+1}_{k=1}\sum^{2k}_{m=0}\frac{\ln(t)^{m}}{t^{k}}F_{k,m}(\xi)+ r^{l-j+1}_F(t,\xi)\right)^j&=\sum_{k=j}^l\sum_{m=0}^{2k}\frac{\ln(t)^{m}}{t^{k}}\Theta^j_{k,m}(\xi)+r'^{l,j}_\Theta(t,\xi)
\end{align*}
with $\Theta^j_{k,m}\in W^{2(N-k),\infty}_\xi$, a product of $j$ factors of $F_{m,k}$-terms, and $r'^{l,j}_\Theta(t)\in W^{2(N-l+j-1),\infty}_\xi$. Finally, collecting all contributions yields
\begin{align*}
\Theta(t,\xi)&=e^{i\ln(t)\lambda_1|\hat{w}_{0,0}(\xi)|^2+F_{0,0}(\xi)}(\sum_{j=0}^l\sum_{k=0}^{2j}\frac{\ln(t)^{k}}{t^{j}}\Theta_{j,k}(\xi))+r^l_{\Theta}(t,\xi)
\end{align*}
with 
\begin{align*}
\Theta_{0,0}(\xi)=1,
\end{align*}
and
\begin{align*}
\Theta_{j,k}(\xi)=\sum_{m=1}^{j}\frac{1}{m!}\Theta^m_{j,k}(\xi),
\end{align*}
if $j\geq1$, and
\begin{align*}
r^l_\Theta(t,\xi)= e^{i\ln(t)\lambda_1|\hat{w}_{0,0}(\xi)|^2+F_{0,0}(\xi)}\sum_{j=1}^lr'^{l,j}_\Theta(t,\xi)+r'^l_{\Theta}(t,\xi),
\end{align*}
which satisfies the required decay estimate. 
\end{proof}
We now deduce the asymptotic expansion of $\hat{f}(t)$.
\begin{proof}[Proof of proposition \ref{propal:bigpropalexpf}]
For all $0\leq l\leq N$, we write
\begin{align*}
\hat{f}(t,\xi)&=\overline{\Theta(t,\xi)}\hat{w}(t,\xi)\\
&=e^{-i\ln(t)\lambda_1|\hat{w}_{0,0}(\xi)|^2-F_{0,0}(\xi)}\left(\sum_{j=0}^l\sum_{k=0}^{2j}\frac{\ln(t)^{k}}{t^{j}}\overline{\Theta_{j,k}(\xi)}+\overline{r^l_{\Theta}(t,\xi)}\right)\left(\sum^{l}_{j=0}\sum^{2j}_{k=0}\frac{\ln(t)^k}{t^j}\hat{w}_{j,k}(\xi)+r^{l}(t,\xi)\right)\\
&=e^{-i\ln(t)\lambda_1|\hat{w}_{0,0}(\xi)|^2-F_{0,0}(\xi)}\left(\sum_{j=0}^l\sum_{k=0}^{2j}\frac{\ln(t)^{k}}{t^{j}}\hat{f}(\xi)_{j,k}\right)+r^l_{f}(t,\xi).
\end{align*}
where
\begin{align}
\hat{f}_{j,k}(\xi)=\sum_{\overrightarrow{\mu},\overrightarrow{\nu}\in\mathbb{N}^2,|\overrightarrow{\mu}|=j,|\overrightarrow{\nu}|=k,\overrightarrow{\nu}_l\leq 2\overrightarrow{\mu}_l}\overline{\Theta_{\overrightarrow{\mu}_1,\overrightarrow{\nu}_1}(\xi)}\hat{w}_{\overrightarrow{\mu}_2,\overrightarrow{\nu}_2}(\xi).
\end{align}
Clearly, proposition \ref{lem:expansionTheta} and proposition \ref{propal:bigpropalexpw} imply that the expansion coefficients satisfy $\hat{f}_{j,k}\in W^{2(N-j),\infty}_\xi$ and the error term satisfies $||r^l_{f}(t)||_{W^{2(N-l),\infty}_\xi}=O(t^{-l-\delta})$ for some $1>\delta>0$. Moreover,
\begin{align*}
\hat{f}_{0,0}(\xi)&=\overline{\Theta_{0,0}(\xi)}\hat{w}_{0,0}(\xi)\\
&=\hat{w}_{0,0}(\xi),
\end{align*}
which ends the proof.
\end{proof} 
\subsection{Asymptotics for the solution}
\label{subsection:asymptoticsolut}
In this section, we derive the asymptotics of the solution $u$ to \eqref{eq:NLS} itself and conclude the proof of the main Theorem \ref{unTheorem:mainTH}. 
\begin{proof}[Proof of Theorem \ref{unTheorem:mainTH}]
From proposition \ref{propal:GWP}, we know that there exists a global solution $u\in C([1,T],H^1)$ which satisfies the sharp decay estimate:
\begin{align*}
t^{1/2}||u(t)||_{L^\infty_x}&\leq ||t^{1/2}u(t)||_{L^\infty_t L^\infty_x}\\
&\lesssim C(\varepsilon_0)
\end{align*}
for all $t\geq 1$, by \eqref{eq:GWPeq}. Therefore, to get the full statement of Theorem \ref{unTheorem:mainTH}, it remains to derive the high-order asymptotic expansion of $u$. \\\\
Recall that the fundamental solution of the linear Schrödinger equation is
\begin{align}
K(t,x)=\left(\frac{1}{2i\pi t}\right)^{1/2}e^{\frac{ix^2}{2t}}.
\end{align}
Hence,
\begin{align*}
u(t,x)&=e^{it\Delta/2}e^{-it\Delta/2}u(t,x)\\
&=e^{it\Delta/2}f(t,x)\\
&=(K(t)\star f(t))(x)\\
&=\int \left(\frac{1}{2i\pi t}\right)^{1/2}e^{\frac{i|x-y|^2}{2t}} f(t,y)dy\\
&=e^{\frac{i|x|^2}{2t}}\left(\frac{1 }{i t}\right)^{1/2}\frac{1}{(2\pi)^{1/2}}\int e^{-i\frac{xy}{t}}( e^{\frac{i|y|^2}{2t}}  f(t,y))dy\\
&=e^{\frac{i|x|^2}{2t}}\left(\frac{1 }{i t}\right)^{1/2}\frac{1}{(2\pi)^{1/2}}\int e^{-i\frac{xy}{t}}(\sum_{j=0}^{N}\frac{1}{j!}\left(\frac{i|y|^2}{2t}\right)^jf(t,y)+\left(e^{\frac{i|y|^2}{2t}}-\sum_{j=0}^{N}\frac{1}{j!}\left(\frac{i|y|^2}{2t}\right)^j\right)f(t,y))dy\\
&=e^{\frac{i|x|^2}{2t}}\left(\frac{1 }{i t}\right)^{1/2}\frac{1}{(2\pi)^{1/2}}\int e^{-i\frac{xy}{t}}(\sum_{j=0}^{N}\frac{1}{j!}\left(\frac{i|y|^2}{2t}\right)^jf(t,y))dy+u_{err.1}^N(t,x).
\end{align*}
We can already state that, for $1/4>\beta>0$, the remainder satisfies  
\begin{align*}
u_{err.1}^N(t,x)&\lesssim\frac{1}{t^{1/2}}\int \left(\frac{|y|^2}{2t}\right)^{N+\beta}f(t,y)dy\\
&\lesssim\frac{1}{t^{N+1/2+\beta}}\int \langle y\rangle^{2(N+\beta)}f(t,y)dy\\
&\lesssim\frac{1}{t^{N+1/2+\beta-\alpha_{2N+1}}}||t^{-\alpha_{2N+1}}\langle y\rangle^{2(N+\beta)+1/2+\varepsilon}f||_{L^\infty_t L^2_{x}}.
\end{align*}
Thus, using the bound $||t^{-\alpha_{2N+1}}f||_{L^\infty_t H^{0,2N+1}_x}<C_0$ of \eqref{eq:GWPeq} in proposition \ref{propal:GWP} leads to
\begin{align*}
u_{err.1}^N(t,x)\lesssim\frac{C_0}{t^{N+1/2+\beta-\alpha_{2N+1}}}.
\end{align*}
This implies that $||u_{err.1}^N(t)||_{L^{\infty}_x}=O(t^{-1/2-N-\delta})$ for some $\delta>0$. On the other hand, 
\begin{align*}
\frac{1}{(2\pi)^{1/2}}\int e^{-i\frac{xy}{t}}(\sum_{j=0}^{N}\frac{1}{j!}\left(\frac{i|y|^2}{2t}\right)^jf(t,y))dy&=\sum_{j=0}^{N}\frac{1}{t^{j}}\frac{i^j}{2^jj!}(-\Delta_\xi)^j\hat{f}\left(t,\frac{x}{t}\right).
\end{align*}
By proposition \ref{propal:bigpropalexpf}, and under the assumption of Theorem \ref{unTheorem:mainTH}, the derivatives $\partial_\xi^{2q}\hat{f}(t)$ admit asymptotic expansions up to order $N-q$ in $L^\infty_\xi$.  Each Laplacian contributes two derivatives and an additional power $\frac{1}{t}$, thus matching the balance between control on the derivatives and order of expansion. To compute $\partial_\xi^{2q}\hat{f}(t)$ explicitly, one must consider derivatives falling on the phase $-\ln(t)F_{1,0}-F_{0,0}$, which is defined in proposition \ref{propal:bigpropalexpf}. For that, we set $F_{0}=\ln(t)F_{1,0}+F_{0,0}$ and apply Faà di Bruno's formula: 
\begin{align}
\partial^{2q}_\xi\hat{f}(t)=e^{-F_{0}}\sum_{l=0}^{2q}\binom{2q}{l}(\sum_{m=0}^lB_{l,m}(-F'_{0},\dots,-F_{0}^{(l-m+1)}))\partial_\xi^{2q-l}\left(\sum_{j=0}^{N-q}\sum_{k=0}^{2j}\frac{\ln(t)^{k}}{t^{j}}\hat{f}_{j,k}\right)+\partial^{2q}r^{N-q}_{f}(t),
\end{align}
where the $B$ represents Bell's polynomials. We decompose Bell's polynomials as follows
\begin{align}
B_{l,m}(x_1+y_1,\dots,x_{l-m+1}+y_{l-m+1})=\sum_{\overrightarrow{\mu},\overrightarrow{\nu}\in\mathbb{N}^{l-m+1},|\overrightarrow{\mu}|,|\overrightarrow{\nu}|\leq m}B_{l,m}^{\overrightarrow{\mu},\overrightarrow{\nu}}\prod_{j=1}^{l-m+1}(x_{j}^{\overrightarrow{\mu}_j}y_{j}^{\overrightarrow{\nu}_j}),
\end{align}
for some coefficients $B_{l,m}^{\overrightarrow{\mu},\overrightarrow{\nu}}$ independent of $\xi$. Gathering logarithm contributions together leads to
\begin{align*}
B_{l,m}(-F'_{0},\dots,-F_{0}^{(l-m+1)})&=\sum_{\overrightarrow{\mu},\overrightarrow{\nu}\in\mathbb{N}^{l-m+1},|\overrightarrow{\mu}|,|\overrightarrow{\nu}|\leq m}B_{l,m}^{\overrightarrow{\mu},\overrightarrow{\nu}}\prod_{j=1}^{l-m+1}((-\ln(t)\partial_\xi^jF_{1,0})^{\overrightarrow{\mu}_j}(-\partial_\xi^jF_{0,0})^{\overrightarrow{\nu}_j})\\
&=\sum_{k=0}^{m}\ln(t)^k\sum_{\overrightarrow{\mu},\overrightarrow{\nu}\in\mathbb{N}^{l-m+1},|\overrightarrow{\mu}|=k,|\overrightarrow{\nu}|\leq m}B_{l,m}^{\overrightarrow{\mu},\overrightarrow{\nu}}\prod_{j=1}^{l-m+1}((-\partial_\xi^jF_{1,0})^{\overrightarrow{\mu}_j}(-\partial_\xi^jF_{0,0})^{\overrightarrow{\nu}_j}).
\end{align*}
We then use the condensate form 
\begin{align*}
B_{l,m}(-F'_{0},\dots,-F_{0}^{(l-m+1)})=\sum_{k=0}^{m}\ln(t)^kb_{l,m,k}(F_{1,0},F_{0,0})
\end{align*}
with 
\begin{align*}
b_{l,m,k}(F_{1,0},F_{0,0})=\sum_{\overrightarrow{\mu},\overrightarrow{\nu}\in\mathbb{N}^{l-m+1},|\overrightarrow{\mu}|=k,|\overrightarrow{\nu}|\leq m}B_{l,m}^{\overrightarrow{\mu},\overrightarrow{\nu}}\prod_{j=1}^{l-m+1}((-\partial_\xi^jF_{1,0})^{\overrightarrow{\mu}_j}(-\partial_\xi^jF_{0,0})^{\overrightarrow{\nu}_j}).
\end{align*}
We can now group together all the $\ln(t)^k$ contributions from the different polynomials and write 
\begin{align*}
\partial^{2q}_\xi\hat{f}(t)&=e^{-F_{0}}\sum_{j=0}^{N-q}\frac{1}{t^{j}}\sum_{l=0}^{2q}\binom{2q}{l}\sum_{m=0}^{l}(\sum_{k_1=0}^{m}\sum_{k_2=0}^{2j}\ln(t)^{k_1+k_2}b_{l,m,k_1}\partial_\xi^{2q-l}\hat{f}_{j,k_2})+\partial^{2q}r^{N-q}_{f}(t),\\
&=e^{-F_{0}}\sum_{j=0}^{N-q}\sum_{k=0}^{2(j+q)}\frac{\ln(t)^k}{t^{j}}\hat{f}^{2q}_{j,k}+\partial^{2q}r^{N-q}_{f}(t),
\end{align*}
where we replace $b_{l,m,k}(F_{1,0},F_{0,0})$ by $b_{l,m,k}$ to lighten the notation and where 
\begin{align}
\hat{f}^{2q}_{j,k}=\sum_{0\leq k_1\leq2q,0\leq k_2\leq2j, |k_1|+|k_2|=k}(\sum_{m=k_1}^{2q}\sum_{l=m}^{2q}\binom{2q}{l}b_{l,m,k_1}\partial_\xi^{2q-l}\hat{f}_{j,k_2}).
\end{align}
Substituting into the previous expansion gives
\begin{align}
\label{eq:exputerrible}
\frac{1}{(2\pi)^{1/2}}\int e^{-i\frac{xy}{t}}(\sum_{j=0}^{N}\frac{1}{j!}\left(\frac{i|y|^2}{2t}\right)^jf(t,y))dy&=e^{-i\ln(t)\lambda_1|\hat{w}_{0,0}\left(\frac{x}{t}\right)|^2-F_{0,0}\left(\frac{x}{t}\right)}\sum_{j=0}^{N}\sum_{k=0}^{2j}\frac{\ln(t)^k}{t^{j}}u_{j,k}\left(\frac{x}{t}\right)+u_{err.2}^N(t,x)
\end{align}
with 
\begin{align*}
u_{j,k}\left(\frac{x}{t}\right)=\sum_{m=0}^j\frac{(-i)^m}{2^mm!}\hat{f}^{2m}_{j-m,k}\left(\frac{x}{t}\right)
\end{align*}
and $||u_{err.2}^N(t)||_{L^{\infty}_x}=O(t^{-1/2-N-\delta})$ for some $\delta>0$.
In particular, the terms $\hat{f}^{2q}_{j,k}$ are defined for $k\leq2(j+q)$, contrary to the terms $\hat{f}_{j,k}$, which are only defined for $k\leq 2j$. Since the terms $\hat{f}^{2q}_{j,k}$ appear with an additional factor $\frac{1}{t^q}$ in the expansion for $u$, no logarithmic powers $\frac{\ln(t)^k}{t^j}$ with $k>2j$ arise in this expansion. Consequently
\begin{align}
u(t,x)=\frac{e^{\frac{ix^2}{2t}-i\lambda_1|u_{0,0}\left(\frac{x}{t}\right)|^2\ln(t))-i\varphi(\frac{x}{2
t})}}{(it)^{1/2}}\sum_{p=0}^N\sum^{2p}_{k}\frac{\ln(t)^k}{t^p}u_{p,k}\left(\frac{x}{t}\right)+u_{err.}(t,x),
\end{align}
for $u_{err.}(t,x)=u_{err.1}^N+u_{err.2}^N$.\\\\
In particular, 
\begin{align*}
u_{0,0}\left(\frac{x}{t}\right)&=\hat{f}_{0,0}\left(\frac{x}{t}\right)\\
&=\hat{w}_{0,0}\left(\frac{x}{t}\right),
\end{align*}
so that the phase correction is exactly $e^{-i\ln(t)\lambda_1|u_{0,0}\left(\frac{x}{t}\right)|^2-i\varphi\left(\frac{x}{t}\right)}$ (for $i\varphi=F_{0,0}$) as stated in Theorem \ref{unTheorem:mainTH}. 
\end{proof}
\appendix
\section{Explicit form of the expansion at order 2}
\label{section:appendix}
In the appendix, we compute the expansion at order 2 ($N=1$) of the solution $u$ to \eqref{eq:NLS}, for fixed parameters $\lambda_1\neq0$ and $\lambda_2\neq0$. At this order, only the cubic and quintic nonlinearities contribute to the asymptotic dynamics. Consequently, the remaining parameters $(\lambda_n)_{n\geq3}$ play no role in this analysis. Our derivation essentially relies on straightforward Taylor expansions and applications of the results established in Section \ref{section:asymptotic}. \\\\
We also compare our expansion with that obtained in \cite{10.1063/1.522967}. To this end, we adopt their notation, which first requires setting  $\lambda_1=-2\alpha$, for $\alpha=\pm1$, and, more importantly, replacing $\Delta/2$ with $\Delta$ in \eqref{eq:NLS}. To account for this difference, we first derive the expansion for the solution $u$ of \eqref{eq:NLS} in our normalization, and then set $v(t,x)=u\left(t,\frac{1}{\sqrt{2}}x\right)$ which solves \eqref{eq:NLS} with $\Delta/2$ replaced by $\Delta$. \\\\
To obtain the explicit form of the expansion for $u$, we just follow the procedure of section \ref{section:asymptotic}, assuming that the initial data satisfy $||u_1||_{H^{1,0}_x}+||f_1||_{H^{0,2N+1}_x}<\varepsilon_0$, for $\varepsilon_0$ sufficiently small, in order to make the expansion rigorous. We begin by deriving the asymptotics of $\hat{w}$. Applying lemma \ref{lem:mainofbasecase} on the base case leads to $\hat{w}=\hat{w}_{0,0}+O(t^{-\delta})$ for some $\delta>0$. To compute the next order explicitly, we use \eqref{eq:organizedexpw} at order 1, which gives
\begin{align*}
 \partial_t\hat{w}(t)&=-i\frac{\lambda_1}{t^2}\Theta(t)\mathfrak{P}^1_1(t)-i\frac{\lambda_2}{t^2}\Theta(t)\mathfrak{P}^2_0(t)-i\Theta(t)\mathfrak{R}^1_2(t)-i\Theta(t)\mathfrak{R}^2_1(t)-i\Theta(t,\xi)\mathfrak{R}^{>2}_0(t)\\
 &=\frac{\lambda_1}{t^2}\Theta(t)\left(2\hat{f}|\partial_\xi\hat{f}|^2+(\partial_\xi\hat{f})^2\overline{\hat{f}}+\hat{f}^2\overline{\partial_\xi^2\hat{f}}\right)(t)-i\frac{\lambda_2}{t^2}\Theta(t)|\hat{f}(t)|^4\hat{f}(t)+O(t^{-2-\delta}).
\end{align*}
Now, as done in lemma \ref{lem:nonlinearityPexactform} we replace $\hat{f}$ by $\overline{\Theta}\hat{w}$ and we simplify the expression with $\overline{\Theta}\Theta=1$. For that, we first set $\chi=-iF$, so that $\hat{w}=e^{i\chi}\hat{f}$, and compute 
\begin{align*}
\Theta\mathfrak{P}^1_1=&i\left(2\hat{w}|\partial_\xi\hat{w}|^2+(\partial_\xi\hat{w})^2\overline{\hat{w}}+\hat{w}^2\overline{\partial_\xi^2\hat{w}}\right)\\
&+2i\hat{w}\left(-\partial_\xi\hat{w}\overline{i\partial_\xi\chi\hat{w}}-\overline{\partial_\xi\hat{w}}i\partial_\xi\chi\hat{w}+\partial_\xi\chi^2|\hat{w}|^2\right)\\
&+i\overline{\hat{w}}\left(-2\partial_\xi\hat{w}i\partial_\xi\chi\hat{w}-\partial_\xi\chi^2\hat{w}^2\right)\\
&+i\hat{w}^2\left(-2\overline{i\partial_\xi\chi\partial_\xi\hat{w}}-\overline{i\partial^2_\xi\chi\hat{w}}-\overline{\partial_\xi\chi\partial_\xi\chi\hat{w}}\right)\\
=&i\left(2\hat{w}|\partial_\xi\hat{w}|^2+(\partial_\xi\hat{w})^2\overline{\hat{w}}+\hat{w}^2\overline{\partial_\xi^2\hat{w}}\right)-\hat{w}|\hat{w}|^2\partial^2_\xi\chi.
\end{align*}
Then, as in lemma \ref{lem:expansionF}, we expand $F=i\int^t_1\frac{\lambda_1|\hat{w}(\tau)|^2}{\tau}d\tau$ (or $\chi=\int^t_1\frac{\lambda_1|\hat{w}(\tau)|^2}{\tau}d\tau$) using the expansion $\hat{w}=\hat{w}_{0,0}+O(t^{-\delta})$. We obtain 
\begin{align*}
 \chi=\ln(t)\nu+\varphi+O(t^{-\delta})
 \end{align*}
  with $\nu=-iF_{0,1}=\lambda_1|\hat{w}_{0,0}|^2$ and $\varphi=-iF_{0,0}=\int^\infty_1\frac{\lambda_1|\hat{w}(\tau,\xi)|^2-|\hat{w}_{0,0}(\xi)|^2}{\tau}d\tau$. Plugging the first-order expansion of $\hat{w}$ and $\chi$ into the previous evolution equation yields
\begin{align*}
    \partial_t\hat{w}(t)=&\frac{1}{t^2}\lambda_1\left(2\hat{w}_{0,0}|\partial_\xi\hat{w}_{0,0}|^2+(\partial_\xi\hat{w}_{0,0})^2\overline{\hat{w}_{0,0}}+\hat{w}_{0,0}^2\overline{\partial_\xi^2\hat{w}_{0,0}}+i\hat{w}_{0,0}|\hat{w}_{0,0}|^2\partial^2_\xi\varphi\right)-i\frac{\lambda_2}{t^2}|\hat{w}_{0,0}|^4\hat{w}_{0,0}\\
    &+\frac{\ln(t)}{t^2}i\lambda_1\hat{w}_{0,0}|\hat{w}_{0,0}|^2\partial^2_\xi\nu,\\
    &+O(t^{-2-\delta})
\end{align*}
which implies
\begin{align}
    \hat{w}(t)&=\hat{w}_{0,0}+\frac{1}{t}\hat{w}_{1,0}+\frac{\ln(t)}{t}\hat{w}_{1,1}+\frac{\ln(t)^2}{t}\hat{w}_{1,2}+O(t^{-1-\delta})
\end{align}
where
\begin{align*}
&\hat{w}_{1,0}=-\lambda_1(2\hat{w}_{0,0}|\partial_\xi\hat{w}_{0,0}|^2+(\partial_\xi\hat{w}_{0,0})^2\overline{\hat{w}_{0,0}}+\hat{w}_{0,0}^2\overline{\partial_\xi^2\hat{w}_{0,0}}+i\hat{w}_{0,0}|\hat{w}_{0,0}|^2\partial^2_\xi\varphi+i\hat{w}_{0,0}|\hat{w}_{0,0}|^2\partial^2_\xi\nu)+i\lambda_2|\hat{w}_{0,0}|^4\hat{w}_{0,0}\\
&\hat{w}_{1,1}=-i\lambda_1\hat{w}_{0,0}|\hat{w}_{0,0}|^2\partial^2_\xi\nu\\
&\hat{w}_{1,2}=0.
\end{align*}
The next step is to obtain the first-order expansion of $\hat{f}$, which requires knowing the expansions of $F$ and $\Theta(t)=e^{F}(t)$.
As in lemma \ref{lem:expansionF}, $F$ is decomposed as follows:
\begin{align*}
F(t)&=i\lambda_1\int^t_1\frac{|\hat{w}(t)|^2}{\tau}d\tau\\
&=\ln(t)F_{0,1}(\xi)+F_{0,0}(\xi)-i\int^\infty_t\frac{\lambda_1(|\hat{w}(\tau,\xi)|^2-|\hat{w}_{0,0}(\xi)|^2)}{\tau}d\tau.
\end{align*}
This implies that 
\begin{align}
F(t)=F_{0,1}\ln(t)+F_{0,0}+\frac{1}{t}F_{1,0}+\frac{\ln(t)}{t}F_{1,1}+\frac{\ln(t)^2}{t}F_{1,2}+O(t^{-1-\beta})
\end{align}
with
\begin{align*}
F_{0,1}&=i\lambda_1|\hat{w}_{0,0}|^2=i\nu\\
F_{0,0}&=i\int^\infty_1\frac{\lambda_1(|\hat{w}(\tau)|^2-|\hat{w}_{0,0}|^2)}{\tau}d\tau=i\varphi\\
F_{1,0}&=-i\lambda_12(\Re(\hat{w}_{0,0}\overline{\hat{w}_{1,0}})+\Re(\hat{w}_{0,0}\overline{\hat{w}_{1,1}})+2\Re(\hat{w}_{0,0}\overline{\hat{w}_{1,2}}))\\
F_{1,1}&=-i\lambda_12(\Re(\hat{w}_{0,0}\overline{\hat{w}_{1,1}})+2\Re(\hat{w}_{0,0}\overline{\hat{w}_{1,2}}))\\
F_{1,2}&=-i\lambda_12\Re(\hat{w}_{0,0}\overline{\hat{w}_{1,2}}).
\end{align*}
Then, applying lemma \ref{lem:expansionTheta} yields
\begin{align*}
\Theta(t)=&e^{F(t)}\\
=&e^{\ln(t)F_{0,1}+F_{0,0}}(\sum_{j=0}^1\frac{1}{j!}\left(F(t)-\ln(t)F_{0,0}-F_{0,1}\right)^j)\\
&+e^{\ln(t)F_{0,1}+F_{0,0}}(e^{F(t,\xi)-\ln(t)F_{0,1}-F_{0,0}}-\sum_{j=0}^1\frac{1}{j!}\left(F(t)-\ln(t)F_{0,1}-F_{0,0}\right)^j)\\
=&e^{i\ln(t)\nu+i\varphi}(1+\frac{1}{t}F_{1,0}+\frac{\ln(t)}{t}F_{1,1}+\frac{\ln(t)^2}{t}F_{1,2})+O(t^{-1-\delta}).
\end{align*}
Using proposition \ref{propal:bigpropalexpf}, we now obtain the expansion of $\hat{f}$ with 
\begin{align*}
\hat{f}(t)&=\overline{\Theta(t)}\hat{w}(t)\\
&=e^{-i\ln(t)\nu-i\varphi}\left(\sum_{j=0}^1\sum_{k=0}^{2j}\frac{\ln(t)^{k}}{t^{j}}\hat{f}_{j,k}\right)+O(t^{-1-\delta}).
\end{align*}
where
\begin{align}
\hat{f}_{j,k}(\xi)=\sum_{\overrightarrow{\mu},\overrightarrow{\nu}\in\mathbb{N}^2,|\overrightarrow{\mu}|=j,|\overrightarrow{\nu}|=k,\overrightarrow{\nu}_l\leq 2\overrightarrow{\mu}_l}\overline{\Theta_{\overrightarrow{\mu}_1,\overrightarrow{\nu}_1}(\xi)}\hat{w}_{\overrightarrow{\mu}_2,\overrightarrow{\nu}_2}(\xi).
\end{align}
that is 
\begin{align*}
\hat{f}_{0,0}&=\hat{w}_{0,0}\\
\hat{f}_{1,0}&=\hat{w}_{1,0}+\hat{w}_{0,0}\overline{F_{1,0}}\\
\hat{f}_{1,1}&=\hat{w}_{1,1}+\hat{w}_{0,0}\overline{F_{1,1}}\\
\hat{f}_{1,2}&=\hat{w}_{1,2}+\hat{w}_{0,0}\overline{F_{1,2}}.
\end{align*}
Finally,  the asymptotics for the solution $u$ itself is computed as in section \ref{subsection:asymptoticsolut} with formula \eqref{eq:exputerrible}:
\begin{align*}
u(t,x)&=e^{\frac{i|x|^2}{2t}}\left(\frac{1 }{it}\right)^{1/2}\frac{1}{(2\pi)^{1/2}}\int e^{-i\frac{xy}{t}}(\sum_{j=0}^{1}\frac{1}{j!}\left(\frac{i|y|^2}{2t}\right)^jf(t,y))dy+O(t^{-3/2-\delta})\\
&=\frac{1 }{(it)^{1/2}}e^{\frac{i|x|^2}{2at}-i\ln(t)\nu\left(\frac{x}{at}\right)-i\varphi\left(\frac{x}{t}\right)}\sum_{j=0}^{1}\sum_{k=0}^{2j}\frac{\ln(t)^k}{t^{j}}u_{j,k}\left(\frac{x}{t}\right)+O(t^{-3/2-\delta})
\end{align*}
with
\begin{align*}
u_{0,0}&=\hat{f}_{0,0}\\
u_{1,0}&=\hat{f}_{1,0}-\frac{i}{2}\Delta_\xi\hat{f}_{0,0}-\frac{\Delta_\xi\varphi}{2}\hat{f}_{0,0}-\frac{2\partial_\xi\varphi}{2a}\partial_\xi\hat{f}_{0,0}+\frac{i(\partial_\xi\varphi)^2}{2}\hat{f}_{0,0}\\
u_{1,1}&=\hat{f}_{1,1}-\frac{\Delta_\xi\nu}{2a}\hat{f}_{0,0}-\frac{2\partial_\xi\nu}{2}\partial_\xi\hat{f}_{0,0}+\frac{2i\partial_\xi\varphi\partial_\xi\nu}{2}\hat{f}_{0,0}\\
u_{1,2}&=\hat{f}_{1,2}+\frac{i(\partial_\xi\nu)^2}{2}\hat{f}_{0,0}.\\
\end{align*}
Combining all previous computations, we express all terms with respect to the coefficients $\hat{w}_{0,0}$, $\varphi$ and $\nu$. Firstly, we write $\hat{w}_{0,0}=e^{i\zeta}\rho$ for $\rho=|\hat{w}_{0,0}|$. Then, 
\begin{align*}
\hat{w}_{1,0}&=-\lambda_1e^{i\zeta}(3\rho(\partial_\xi\rho)^2+\rho^2\partial^2_\xi\rho-i(\partial_\xi^2\zeta+\partial_\xi^2\varphi)\rho^3+i\partial^2_\xi\nu\rho^3)+i\lambda_2e^{i\zeta}\rho^5\\
\hat{w}_{1,1}&=-i\lambda_1e^{i\zeta}\rho^3\partial^2_\xi\nu\\
\hat{w}_{1,2}&=0\\
F_{0,1}&=i\nu\\
F_{0,0}&=i\varphi\\
F_{1,0}&=i\lambda_1^2(3\rho^2(\partial_\xi\rho)^2+\rho^3\partial^2_\xi\rho)\\
F_{1,1}&=0\\
F_{1,2}&=0\\
\hat{f}_{0,0}&=e^{i\zeta}\rho\\
\hat{f}_{1,0}&=-\lambda_1e^{i\zeta}(3\rho(\partial_\xi\rho)^2+\rho^2\partial^2_\xi\rho-i(\partial_\xi^2\zeta+\partial_\xi^2\varphi)\rho^3+i\partial^2_\xi\nu\rho^3)+i\lambda_2e^{i\zeta}\rho^5-ie^{i\zeta}\lambda_1^2(3\rho^3(\partial_\xi\rho)^2+\rho^4\partial^2_\xi\rho)\\
\hat{f}_{1,1}&=-i\lambda_1e^{i\zeta}\rho^3\partial^2_\xi\nu\\
\hat{f}_{1,2}&=0\\
u_{0,0}&=e^{i\zeta}\rho\\
u_{1,0}&=-\lambda_1e^{i\zeta}(3\rho(\partial_\xi\rho)^2+\rho^2\partial^2_\xi\rho-i(\partial_\xi^2\zeta+\partial_\xi^2\varphi)\rho^3+i\partial^2_\xi\nu\rho^3)+i\lambda_2e^{i\zeta}\rho^5-ie^{i\zeta}\lambda_1^2(3\rho^3(\partial_\xi\rho)^2+\rho^4\partial^2_\xi\rho)\\
&+e^{i\zeta}(-i\Delta_\xi\rho-\frac{(\partial_\xi^2\zeta+\partial_\xi^2\varphi)}{2}\rho-\frac{2(\partial_\xi\zeta+\partial_\xi\varphi)}{2}\partial_\xi\rho-\frac{i(\partial_\xi\zeta+\partial_\xi\varphi)^2}{2}\rho)\\
u_{1,1}&=-i\lambda_1e^{i\zeta}\rho^3\partial^2_\xi\nu+e^{i\zeta}(-\frac{\partial_\xi^2\nu}{2}\rho-\frac{2\partial_\xi\nu}{2}\partial_\xi\rho+\frac{2i(\partial_\xi\zeta+\partial_\xi\varphi)\partial_\xi\nu}{2}\rho)\\
u_{1,2}&=e^{i\zeta}\frac{i(\partial_\xi\nu)^2}{2}\rho.
\end{align*}
Then, to match \cite{10.1063/1.522967}, we write 
\begin{align}
v(t,x):=u(t,\frac{1}{\sqrt{2}}x)=\frac{1 }{t^{1/2}}e^{\frac{i|x|^2}{2t}-i\ln(t)\nu\left(\frac{x}{t}\right)}\sum_{j=0}^{1}\sum_{k=0}^{2j}\frac{\ln(t)^k}{t^{j}}e^{g\left(\frac{x}{t}\right)}v_{j,k}\left(\frac{x}{t}\right)+O(t^{-3/2-\delta}),
\end{align}
for $g\left(\frac{x}{t}\right)=(-\zeta\left(\frac{x}{\sqrt{2}t}\right)-\varphi\left(\frac{x}{\sqrt{2}t}\right)-\pi/4)$. Then, $\nu=\lambda_1\rho^2$ leads to
\begin{align*}
v_{0,0}=&h\\
v_{1,0}=&-2\lambda_1(3h(\partial_\xi h)^2+h^2\partial^2_\xi h)+\partial_\xi^2gh+2\partial_\xi g\partial_\xi h\\
&-2\lambda_1i\partial_\xi^2gh^3-i\lambda_1^2(10\partial_\xi h\partial_\xi h+6h\partial^2_\xi h) h^3+i\lambda_2h^5+2(-i\Delta_\xi h-\frac{i(\partial_\xi h)^2}{2}h)\\
v_{1,1}=&-2\lambda_1(3h(\partial_\xi h)^2+h^2\partial^2_\xi h)-4i\lambda_1^2h^3(\partial_\xi h\partial_\xi h+h\partial^2_\xi h)-4\lambda_1i\partial_\xi gh^2\partial_\xi h\\
v_{1,2}=&i(2\lambda_1h\partial_\xi h)^2h,\\
\end{align*}
where $h\left(\frac{x}{t}\right)=\rho\left(\frac{x}{\sqrt{2}t}\right)$. Here $h$ denotes $f$ in the notation of \cite{10.1063/1.522967}. This choice avoids confusion with the profile $f$ of the present paper. For $h$ and $g$ two given functions and for $\lambda_1=-2\alpha$ where $\alpha=\pm1$, the coefficients in the sense of \cite{10.1063/1.522967} are
\begin{align*}
h_{1,0}&:=\Re(v_{1,0})=4\alpha(3h(\partial_\xi h)^2+h^2\partial^2_\xi h)+\partial_\xi^2gh+2\partial_\xi g\partial_\xi h\\
h_{1,1}&:=\Re(v_{1,1})=4\alpha(3h(\partial_\xi h)^2+h^2\partial^2_\xi h)\\
h_{1,2}&:=\Re(v_{1,2})=0\\
\theta_{2,0}h&:=\Im(v_{1,0})=-2\lambda_1\partial_\xi^2gh^3-\lambda_1^2(10\partial_\xi h\partial_\xi h+6h\partial^2_\xi h) h^3+2(-\Delta_\xi h-\frac{(\partial_\xi h)^2}{2}h)+\lambda_2h^5\\
\theta_{2,1}h&:=\Im(v_{1,1})=-16h^3(\partial_\xi h\partial_\xi h+h\partial^2_\xi h)+8\alpha\partial_\xi gh^2\partial_\xi h\\
\theta_{2,2}h&:=\Im(v_{1,2})=16(h\partial_\xi h)^2h.
\end{align*}
The coefficients $h_{1,0}$, $h_{1,1}$, $h_{1,2}$, $\theta_{2,1}$ and $\theta_{2,2}$ can be compared with those obtained in \cite{10.1063/1.522967}. The coefficient $\theta_{2,0}$ is not computed in \cite{10.1063/1.522967}. Here, this term is the only one that contains the contribution of the quintic nonlinearity, through the expression $\lambda_2h^5$. This quintic term appears at order $t^{-1}$ in the phase as a vanishing correction. In contrast, the cubic term $-2\alpha h^3$ (or, more generally, $\lambda_1 h^3$) appears multiplied by a logarithmic factor $\ln(t)$ and therefore induces long-range effects. 
\printbibliography
\end{document}